\documentclass[12pt]{article} 

\usepackage{amsmath} 
\usepackage{amsthm}
\usepackage{amssymb}
\usepackage{mathtools}
\usepackage[hidelinks]{hyperref}
\usepackage[utf8]{inputenc}
\usepackage{csquotes}
\usepackage[pagewise]{lineno}
\usepackage[english]{babel}
\usepackage{enumitem}
\usepackage{titlesec}
\usepackage{bm}
\usepackage[title]{appendix}

\usepackage{geometry}
\makeatletter
\BeforeBeginEnvironment{proof}{%
  \def\@Opargbegintheorem#1#2#3#4{#4\trivlist
      \item[]{#3#2\@thmcounterend\ }}%
}
\AfterEndEnvironment{proof}{%
  \def\@Opargbegintheorem#1#2#3#4{#4\trivlist
      \item[\hskip\labelsep{#3#1}]{#3(#2)\@thmcounterend\ }}%
}
 \newtheorem{thm}{Theorem}[subsection]
 \newtheorem{lem}[thm]{Lemma}
 \newtheorem{defn}{Definition}[subsection]
 \newtheorem{exmp}{Example}[subsection]
 \newtheorem{rema}{Remark}[subsection]

\def\rem#1#2{ #1\;\textnormal{rem}\;#2}

\def\res#1#2{ \textnormal{eval}\left(#1;\  #2\right)}
\providecommand{\keywords}[1]
{
	\small	
	\textbf{\text{Keywords:}} #1
}

\providecommand{\subclass}[1]
{
	\small	
	\textbf{\text{MSC(2010):}} #1
}

\title{\Large An Algebraic Approach to $q$-Partial Fractions \\ and Sylvester Denumerants \thanks{Dedicated to Bhagawan Sri Sathya Sai Baba.}}

\author{\normalsize N. Uday Kiran\thanks{nudaykiran@sssihl.edu.in}  \\
	\small Department of Mathematics and Computer Science\\
	\small Sri Sathya Sai Institute of Higher Learning, Puttaparthi, India \\
}

\date{}

\begin{document}
    \maketitle
	\begin{abstract}
		In 1857 Sylvester established an elegant theory that certain counting functions (which he termed denumerants) are quasi-polynomials by decomposing them into periodic and non-periodic parts. Each component of the decomposition, called a Sylvester wave, corresponds to a root of unity. Recently several researchers, using either combinatorial arguments or complex analytic techniques, obtained explicit formulas for the waves. In this work, we develop an algebraic approach to the Sylvester's theory. Our methodology essentially relies on deriving $q$-partial fractions of the generating functions of the denumerants, and thereby obtain new explicit formulas for the waves. The formulas we obtain are expressed in terms of the degenerate Bernoulli numbers and a generalization of the Fourier-Dedekind sum. Further, we also prove certain reciprocity theorems of the generalized Fourier-Dedekind sums and a structure result on the top-order terms of the waves. The proofs rely on our evaluation operator and our far-reaching generalization of the Heaviside's cover-up method for partial fractions. SageMath code for this work is available in the public domain.\\
		\\
		\\
		\subclass{26C15 $\cdot$  11P81 $\cdot$ 11F20}
	\end{abstract}

Restricted partitions of positive integers have been studied extensively starting from the work of Euler on generating functions. So as to deploy binomial expansions, Cayley took the approach of partial fraction decomposition of these generating functions. Subsequently, Sylvester generalized the restricted partitions by formulating the counting problem as the number of non-negative integral solutions to certain linear Diophantine equations. This generalization, which he termed denumerants, permits one to define partitions into specified parts - repeated or not. 

Given a $r$-tuple $\textbf{A}=(a_{1},\cdots, a_{r})$ of positive integers with $\textnormal{gcd}(a_{1},\cdots,a_{r})=1$, the denumerant of $t$ with parts from $\textbf{A}$, denoted $d(t;\textbf{A})$, is the cardinality of the set 
$$
\left\{(x_{1},\cdots,x_{r}): \sum^{r}_{j=1}x_{j}a_{j}=t \textnormal{ and } \  x_{i}\geq 0,  \textnormal{ for } i=1,\cdots,r  \right\}.
$$
It can be easily seen that $d(t;\textbf{A})$ is given by the generating function  
\begin{equation}\label{gen_fun}
\frac{1}{(1-x^{a_{1}})\cdots (1-x^{a_{r}})}=\sum^{\infty}_{t=0}d(t; \textbf{A})x^{t}.
\end{equation}
The important case of distinct $a_{1},\cdots,a_{r}$ corresponds to the restricted partition function associated with the Frobenius problem \cite{Beck,Alfonsin}.

In \cite{Sylvester}, Sylvester made a remarkable progress by showing that 
$
 d(t;\textbf{A}) = \sum W_{j}(t; \textbf{A}),  
$ where the sum is over divisors $j$ of the entries in the  tuple $\textbf{A}$. The term $W_{j}(t;\textbf{A})$, called a Sylvester wave, is a quasi-polynomial in $t$ given by the  
coefficient of $z^{-1}$, i.e. the residue at $z=0$ of the function 
$$
\sum_{\stackrel{0\leq n<j}{\textnormal{gcd}(n,j)=1}}\frac{\xi_{j}^{-nt}e^{zt}}{(1-\xi_{j}^{a_{1}n}e^{-a_{1}z})\cdots(1-\xi_{j}^{a_{r}n}e^{-a_{r}z})},
$$
where $\xi_{j}=e^{2\pi i/j}$. Moreover, $\xi_{1}=1$ so that $W_{1}(t;\textbf{A})$ is a polynomial in $t$. Glaiser \cite{Glaiser} leveraged these waves to perform extensive hand calculations of restricted partitions (see \cite{Dickson,Agnarsson} for history). 

Building upon Cayley's work, Munagi \cite{munagi2} formulated a non-standard approach to partial fractions, the so-called $q$-partial fractions. The denominators of $q$-partial fractions are of the form $(1-x^{r})^{s}$, for some integers $r$ and $s$, so that they provide us direct representation formulas. Sills and Zeilberger \cite{Sills} obtained explicit expressions for $q$-partial fractions for partition function of $n$ into at most $m$ parts using Mapel computer algebra system. Their method involved ``rigorous guessing'' by a quasi-polynomial ansatz.   

In this work, we use elementary algebraic methods to obtain explicit $q$-partial fraction formulas for (\ref{gen_fun}), and thereby obtain formulas for the denumerants. These formulas can be handled by a computer algebra system \cite{uk_sagemath}. Our methodology involves a symbolic evaluation operator based on commutative algebra and a rigorous higher degree extension of Heaviside's cover-up method, given in \cite{man}. We also aim to show that the original idea of Cayley-Sylvester not only has a computational advantage but, to borrow the words of Hardy-Ramanujan \cite[pg. no. 76]{Ramanujan}, the interest which they possess is algebraical. 

Usually, computing a $q$-partial fraction for (\ref{gen_fun}) is quite complicated and messy (for some examples see \cite[Section 4.3]{Alfonsin}). Therefore, in this work, we write (\ref{gen_fun}) as
\begin{equation}\label{transformed_new}
\frac{1}{(1-x^{a_{1}})\cdots(1-x^{a_{r}})}=\frac{p(x)}{(1-x)^{m}(1-x^{n_{1}})^{r_{1}}\cdots(1-x^{n_{k}})^{r_{k}}},
\end{equation}
where $n_{1}, \cdots,n_{k}$ are pairwise relatively prime divisors of the elements in $\textbf{A}$ with multiplicity $r_{1},\cdots,r_{k}$ respectively. Furthermore, the right hand side of (\ref{transformed_new}) is of independent interest as it is a generalization of the restricted partition associated with the Ehrhart polynomial \cite{Beck,Beck3}.

It is easy to see that $p(x)$ in (\ref{transformed_new}) can be optimally chosen so that we retain as many divisors as possible. Although our simplification completely characterize the polynomial part $W_{1}(t;\textbf{A})$ and the top-order terms of the waves it is naturally deficient in obtaining explicit expressions for `all the frequencies' corresponding to all the divisors. For example, the terms associated with $2$ and $4$ are merged while simplifying  
\begin{equation}\label{exmp_5_case}
\prod^{5}_{j=1}\frac{1}{1-x^{j}}=\frac{1+x^{2}}{(1-x)(1-x^{3})(1-x^{4})^{2}(1-x^{5})}.
\end{equation}

The main results of this paper uncover the beautiful structure of the denumerants associated with the right hand side of (\ref{transformed_new}). We obtain explicit formula for $q$-partial fraction for (\ref{transformed_new}) of the form
\begin{equation}\label{intro_qpf}
\frac{p(x)}{(1-x)^{m}(1-x^{n_{1}})^{r_{1}}\cdots(1-x^{n_{k}})^{r_{k}}}=\frac{g_{0}(x)}{(1-x)^{s}}+\sum_{j=1}^{k}\frac{g_{j}(x)}{(1-x^{n_{j}})^{r_{j}}},
\end{equation}
where $s=m+r_{1}+\cdots+r_{k}$. With (\ref{intro_qpf}) at hand we can obtain precise and explicit formulas for the terms of the decomposition 
\begin{equation}\label{decomposition}
d(t;\textbf{A})=W_{1}(t;\textbf{A})+\sum^{k}_{j=1}W_{n_{j}}(t;\textbf{A}),
\end{equation}
where $W_{b}(t;\textbf{A})$ is the wave corresponding to $b^{th}$ root of unity. We also show that $W_{n_{j}}(t;\textbf{A})$ is a quasi-polynomial in $t$ of the form 
$
W_{n_{j}}(t;\textbf{A})=\alpha^{(j)}_{r_{j}-1}(t)t^{r_{j}-1}+\cdots+\alpha^{(j)}_{0}(t),
$ 
where $\alpha^{(j)}_{i}(t)$ is a periodic function in $t$ with periodicity $n_{j}$ and $r_{j}$ is the multiplicity of $n_{j}$. Further, we bring out the resemblance of the top-order coefficient of $W_{n_{j}}(t;\textbf{A})$ to certain generalization of the Fourier-Dedekind sums and prove a corresponding reciprocity theorem. 

Recently, obtaining explicit formulas of the waves received a lot of attention in the literature. An explicit expression for $W_{1}(t;\textbf{A})$ using Bernoulli numbers was given by Beck, Gessel and Komatsu \cite{Beck0} and without Bernoulli numbers by Dilcher and Vignat \cite{Dilcher}. Rubinstein and Fel \cite{Rubin} derived formulas for $W_{j}(t;\textbf{A})$ using Bernoulli and Euler polynomials of higher order. Rubinstein \cite{Rubinstein1} further simplified the formulas only with Bernoulli polynomials. O'Sullivan \cite{Sullivan} studied the relationship with Rademacher's coefficients. Cimpoea\c{s} and Nicolae \cite{Mircea} gave direct formula for $d(t;\textbf{A})$. In \cite{BB} the authors gave an algorithm for computing the top-order terms of $d(t;\textbf{A})$ using poset structures. 

The paper is organized as follows. The main results of the paper and our methodology are stated in Section \ref{main_results}. The extended cover-up method is proved in Section \ref{sec_PFs}. The polynomial $\Psi_{m}(x)$ and its relation with the degenerate Bernoulli number and the Fourier series which play a crucial role in our $q$-partial fractions are discussed in Section \ref{sec_psi}. The preparation lemmas for $q$-partial fractions are proved in Section \ref{di_lemma}. Subsequently, in Section \ref{subsec_qpf}, the derivations of the $q$-partial fractions and the denumerants are given. Section \ref{sec_recip} discusses a generalization of the Fourier-Dedekind sum and Section \ref{sylvester_waves} is on Sylvester waves.

\section{Main Results and Our Methodology}\label{main_results}

\subsection{Main Results} 

Our first main result is deriving an explicit formula for the polynomial part (the first wave, $W_{1}(t;\textbf{A})$). For other explicit formulas for $W_{1}(t; \textbf{A})$ see \cite{Beck0,Dilcher,Mircea}.

\begin{thm}[Polynomial Part]

\begin{equation}\label{W1}
W_{1}(t;\textbf{A})=\frac{1}{a_{1}\cdots a_{r}}\sum^{r-1}_{j=0}\sum_{j_{1}+\cdots+j_{r}=j}(-1)^{j}\begin{pmatrix}t+r-j-1 \\ t\end{pmatrix}\frac{\tilde{\beta}_{j_{1}}(a_{1})\cdots \tilde{\beta}_{j_{r}}(a_{r})}{j_{1}!\cdots j_{r}!},
\end{equation}
where $\tilde{\beta}_{j}(a)$ is the reciprocal degenerate Bernoulli number (see Section \ref{sec_psi}). 
\end{thm}

Degenerate Bernoulli numbers were proposed by Carlitz \cite{degenerate_bernoulli}, and are of contemporary interest (see \cite{degenerate_bernoulli,young,Zhang}). The key to our investigation is our observation that $\tilde{\beta}_{k}(m)$ is given by the exponential generating function  
$$
\frac{m(1-x)}{1-x^{m}}=\frac{m}{\Psi_{m}(x)}=\sum_{k=0}^{\infty}(-1)^{k}\frac{\tilde{\beta}_{k}(m)}{k!}(1-x)^{k}.
$$
As we will see later the polynomial $\Psi_{m}(x)=\sum_{j=0}^{m-1}x^{j}$ plays a fundamental role in our $q$-partial fractions. Moreover, by Lemma \ref{pow_k} we define a variant of the degenerate Bernoulli polynomial $\tilde{\beta}_{k}(m,x)$ relying on a residue of powers 
\begin{equation}\label{RDB}
\frac{(-1)^{k}}{k!}\tilde{\beta}_{k}(m,x) = \rem{\left(-\frac{1}{m}x\Psi'_{m}(x)\right)^{k}}{\Psi_{m}(x)},
\end{equation}
where rem stands for the polynomial remainder operator. By setting $x=1$ we obtain $\tilde{\beta}_{k}(m)=\tilde{\beta}_{k}(m,1)$. This approach is new in the literature and provides us with an $O(mn)$ algorithm to compute the first $n$ degenerate Bernoulli numbers and $O(m\log n)$ to directly compute the $n^{th}$ number. Several interesting properties of (\ref{RDB}) are explored in \cite{uksl}.

Our next main result is a $q$-partial fraction on the right hand side of (\ref{transformed_new}).
\begin{thm}[$q$-Partial Fraction Decomposition]\label{main_qpf}
The following $q$-partial fraction holds
\begin{eqnarray}\label{main_qpf_formula}
\frac{p(x)}{(1-x)^{m}}\prod_{j=1}^{k}\frac{1}{(1-x^{n_{j}})^{r_{j}}}&=&\sum^{m+s-1}_{j=0}\frac{c_{j}}{(1-x)^{m+s-j}} + \sum^{k}_{j=1}\frac{g_{j}(x)}{(1-x^{n_{j}})^{r_{j}}}\label{rational_fun1}\\ 
\nonumber &=& \sum^{m+s-1}_{j=0}\frac{c_{j}}{(1-x)^{m+s-j}} + \sum^{k}_{j=1}\sum_{i=0}^{r_{j}-1}\frac{g_{ij}(x)}{(1-x^{n_{j}})^{r_{j}-i}},
\end{eqnarray}
where $s=r_{1}+\cdots+r_{k}$ and $\textnormal{deg}(g_{ij})<n_{j}$,  
$$
c_{j}=\frac{1}{n_{1}^{r_{1}}\cdots n_{k}^{r_{k}}}\sum_{i_{0}+i_{1}+\cdots+i_{k}=j}(-1)^{j}\frac{D^{(i_{0})}p(1)}{i_{0}!}\frac{\tilde{\beta}^{(r_{1})}_{i_{1}}(n_{1})\cdots \tilde{\beta}^{(r_{k})}_{i_{k}}(n_{k})}{i_{1}!\cdots i_{k}!},
$$
for  $0\leq j\leq m+k-1$, $D$ is differentiation, and 
\begin{equation}\label{gj_term_main_qpf}
g_{j}(x) =  (1-x)^{r_{j}}\res{\frac{p(x)}{(1-x)^{m+s}}\prod_{\stackrel{i=1}{i\neq j}}^{k}\frac{1}{\Psi_{n_{i}}(x)^{r_{i}}}}{\Psi_{n_{j}}(x)^{r_{j}}},
\end{equation}
where $\textnormal{deg}(g_{j})<n_{j}r_{j}$. Here $\tilde{\beta}_{k}^{(r)}(n)$ is order $r$ reciprocal degenerate Bernoulli number given in Section \ref{sec_psi}. 
\end{thm}

A simplification of $g_{j}(x)$ is done by the eval function (explained in the next subsection). The expression for $g_{j}(x)$ is obtained by direct formulas given in Table \ref{eval_table} (in Subsection \ref{our_methodology}) and basic properties of eval operator are given in the Lemma \ref{lem_res}. In particular, for $r_{1}=\cdots=r_{k}=1$ the expression for $g_{j}(x)$ simplifies to  
$$
g_{j}(x) =  (1-x)\rem{\left(q_{j}(x)^{m+k}p(x)\sum^{a_{1}^{(j)}}_{i_{1}=0}\cdots\sum^{a_{k}^{(j)}}_{i_{k}=0}x^{\sum_{s=1}^{k}i_{s}n_{s}\%n_{j}}\right)}{\Psi_{n_{j}}(x)},
$$
where $q_{j}(x)=-\frac{1}{n_{j}}x\Psi'_{n_{j}}(x)$, $ a_{j}^{(j)}=0$ and $(a_{i}^{(j)}+1)n_{i}-b_{i}^{(j)}n_{j}=1, a^{(j)}_{i}, b^{(j)}_{i}\geq0$ for $i\neq j$. For this special case, denoting $g_{j}(x)=\sum^{n_{j}-1}_{i=0}c_{i}^{(j)}x^{i}$ we obtain a representation formula for the denumerant
$$
d(t;\textbf{A})=\sum^{m+k-1}_{j=0}c_{j}
\begin{pmatrix}
t+m+k-j-1\\
t\\
\end{pmatrix}
+\sum^{k}_{j=1}\frac{1}{n_{j}}c^{(j)}_{t\%n_{j}}.
$$
This formula is easy to compute as one can perform a remainder of $\Psi_{m}(x)$ in linear time using the rule given in Lemma \ref{cor_phi}. Further, by the FFT algorithm one can obtain the expression for $g_{j}(x)$ in $O(n_{j}\log n_{j})$ time for $j=1,\cdots,k$.

Performing a Fourier analysis on the $q$-partial fraction formula (\ref{main_qpf_formula}) we obtain the decomposition 
$$
d(t;\mathbf{A})=W_{1}(t;\mathbf{A})+W_{n_{1}}(t;\mathbf{A})+\cdots+W_{n_{k}}(t;\mathbf{A})
$$
and derive a formula for the top-order term of the quasi-polynomial $W_{n_{j}}(t;\textbf{A})$, for $1\leq j\leq k$. The term $W_{n_{j}}(t;\textbf{A})$ corresponds to the periodic part associated with the frequency of $n_{j}$ given by 
$$
W_{n_{j}}(t;\textbf{A})= \alpha^{(n_{j})}_{r_{j-1}}(t)t^{r_{j}-1}+\cdots+\alpha^{(n_{j})}_{0}(t),
$$
where $\alpha^{(n_{j})}_{i}(t)$ is a periodic function in $t$ with periodicity $n_{j}$, for $i=0,\cdots,(r_{j}-1)$ and $r_{j}$ is the multiplicity of $n_{j}$. Our next main result is to provide an explicit formula for the top-order term's coefficient $\alpha_{r_{j-1}}^{(n_{j})}(t)$.

\begin{thm}[Structure Theorem]\label{structure_theorem}
The denumerant $d(t;\textbf{A})$ can be decomposed as (\ref{decomposition}) with $W_{1}(t;\textbf{A})$ given in (\ref{W1}) and $W_{n_{j}}(t;\textbf{A})$ is a quasi-polynomial of degree $r_{j}-1$ given by 
\begin{equation}\label{wave_j}
W_{n_{j}}(t;\textbf{A})=P_{j}(\xi_{j})\frac{t^{r_{j}-1}}{(r_{j}-1)!}+O(t^{r_{j}-2}), 
\end{equation}
where 
$$
P_{j}(\xi_{j})=\frac{1}{n_{j}^{r_{j}}}\sum^{r_{j}-1}_{\alpha=1}\frac{p(\xi^{\alpha}_{j})\xi_{j}^{-\alpha t}}{(1-\xi_{j}^{\alpha})^{m}(1-\xi_{j}^{\alpha n_{1}})^{r_{1}}\cdots \widehat{(1-\xi_{j}^{\alpha n_{j}})^{r_{j}}}\cdots(1-\xi^{\alpha n_{k}}_{j})^{r_{k}}},
$$
for $\xi_{j}=e^{2\pi i/n_{j}}$, $j=1,\cdots,k, \textnormal{ and } \widehat{ }$ stands for dropping the term. 
\end{thm}

As a consequence of our structure theorem one can directly read-off the top-order terms of the periodic parts of the denumerants.

\begin{exmp}
The restricted partition function $p_{5}(t)$ associated with (\ref{exmp_5_case}) can be written as\\ $p_{5}(t)=W_{1}(t)+W_{3}(t)+W_{4}(t)+W_{5}(t)$ where 
\begin{eqnarray}
\nonumber W_{3}(t)&=&\frac{1}{3}\sum_{j=1}^{2}\frac{\xi_{3}^{-jt}}{(1-\xi_{3}^{j})(1-\xi_{3}^{2j})(1-\xi_{3}^{4j})(1-\xi_{3}^{5j})},\\
\nonumber W_{5}(t)&=&\frac{1}{5}\sum_{j=1}^{4}\frac{\xi_{5}^{-jt}}{(1-\xi_{5}^{j})(1-\xi_{5}^{2j})(1-\xi_{5}^{3j})(1-\xi_{5}^{4j})},\quad \textnormal{ and }\\
\nonumber W_{4}(t)&=& \left(\frac{1}{4}\sum^{3}_{j=1}\frac{\xi_{4}^{-jt}(1+\xi_{4}^{2j})}{(1-\xi_{4}^{j})(1-\xi_{4}^{3j})(1-\xi_{4}^{5j})}\right)t+O(1)=\frac{(-1)^{t}}{2^{6}}t+O(1).
\end{eqnarray}
The $O(1)$ terms of $W_{4}(t)$ can be computed directly by Equation \ref{qpf_case2_eqn}, Theorem \ref{taylor_Dm} and Theorem \ref{eval_hj}. 
\end{exmp}


Another key finding of this paper is that the top-order term of $W_{n_{j}}(t;\textbf{A})$ has a similarity to a  generalization of the Fourier-Dedekind sum. In fact, for the special case $r_{1}=\cdots=r_{k}=r$ we can write 
$$
W_{n_{j}}(t;\textbf{A})=S^{(m,r,p(x))}_{-t}(\textbf{A}_{j};n_{j})\frac{t^{r-1}}{(r-1)!}+O(t^{r-2}),
$$
where $\textbf{A}_{j}=(n_{1},\cdots, \widehat{n_{j}},\cdots,n_{k})$ and $S^{(m,r,p(x))}_{t}$ is the generalized Fourier-Dedekind sum (given in Definition \ref{FDS}). This motivates us to generalize the reciprocity results given in  \cite{Carlitz1,Gessel,Zagier,Beck,Beck1,Tuskerman}.

\begin{figure}[t!]
\centering 
\includegraphics[scale=0.6]{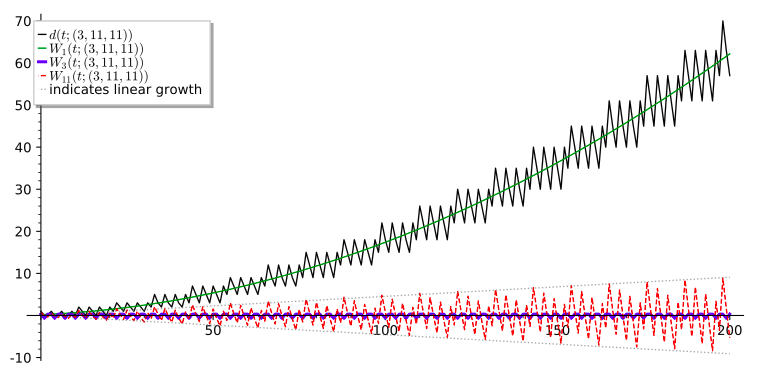}
\\ The multiplicity of $11$ and $3$ are $2$ and $1$ respectively so we have $W_{11}(t)=\alpha_{1}^{(11)}(t)t+\alpha_{0}^{(11)}(t)$ and $W_{3}(t)=\alpha_{0}^{(3)}(t)$, where $\alpha_{1}^{(11)}(t), \alpha_{0}^{(11)}(t)$ and $\alpha_{0}^{(3)}(t)$ are periodic functions. Thus the amplitude of the periodic part $W_{11}(t,\mathbf{A})$ grows linearly in $t$. Here $d(t,\mathbf{A})=W_{1}(t,\mathbf{A})+W_{3}(t,\mathbf{A})+W_{11}(t,\mathbf{A})$ for $\mathbf{A}=(3,11,11)$. 
\end{figure}

\begin{thm}[Reciprocity Theorem] The following result holds
$$
\sum^{k}_{j=1}S_{t}^{(m,r,p(x))}(n_{1},\cdots,\widehat{n_{j}},\cdots,n_{k};n_{j}) =-\textnormal{poly}(-t),\quad \textnormal{for } 1\leq t< \lambda,
$$
where the polynomial $\textnormal{poly}(-t)$ and $\lambda$ will be made precise later. 
\end{thm}

As a consequence of the reciprocity theorem and adopting our algebraic tools, we provide an algebraic proof for (Problem E 3339, Duran \cite{Duran}) in Subsection \ref{computation}. An analytic proof was by Gessel \cite{Gessel}.

\subsection{Our Methodology}\label{our_methodology}
We first create an algebraic framework by generalizing the  Heaviside's cover-up method. Given distinct numbers $c_{1},\cdots, c_{k}$ and $h(x)$ a polynomial with $\textnormal{deg}(h)<k$, suppose the expression for the partial fraction is 
\begin{equation}\label{heaviside}
h(x)\prod_{j=1}^{k}\frac{1}{x-c_{j}}=\sum_{j=1}^{k}\frac{C_{j}}{x-c_{j}},
\end{equation}
then the cover-up method suggests: to obtain $C_{i}$ we need to evaluate the left hand side of (\ref{heaviside}) at $x=c_{i}$ seemingly `covering' the $(x-c_{i})$ factor. Thus, we obtain 
$
C_{i}=h(c_{i})\prod^{k}_{j=1,j\neq i}\frac{1}{c_{i}-c_{j}}.$ 

Heaviside method is given only for the linear case. Man \cite{man} provided an iterative procedure, based on long division and substitution, to extend the method to the quadratic case. His method does not require factorization into linear factors, differentiation or solving simultaneous equations. This paper gives a rigorous extension to higher degree polynomials the method given in \cite{man}. We rely on localization of rings, substitution rule associated with the generalized Remainder Theorem and the B\'{e}zout's identity. 

The key to rigorously extend this method to higher degree polynomials is our symbolic evaluation operator based on the principles of localization and substitution. Given non-constant polynomials $r(x), s(x), a(x)\in \mathbb{Q}[x]$ and $s(x),a(x)$ are relatively prime, we define the evaluation of the fraction $\frac{r(x)}{s(x)}$ with respect to $a(x)$ by
$$
\res{\frac{r(x)}{s(x)}}{a(x)} = \rem{(\alpha(x)r(x))}{a(x)},
$$ 
where $\alpha(x) s(x)\equiv 1 \textnormal{ mod }a(x)$ and $\textnormal{rem}$  is the polynomial remainder operator. Our extension result is stated below and proved in Section \ref{sec_PFs}.
\begin{thm}[Extended Cover-Up Method]\label{cover_up}
Let $p_{1}(x),\cdots, p_{n}(x)\in \mathbb{Q}[x]$ be pairwise relatively prime polynomials and let $f(x)\in \mathbb{Q}[x]$ satisfying $deg(f)< deg(p_{1}\cdots p_{k})$. Then, we have the following identity 
$$
f(x)\prod_{j=1}^{k}\frac{1}{p_{j}(x)}=\sum^{k}_{j=1}\frac{\res{f(x)/g_{j}(x)}{p_{j}(x)}}{p_{j}(x)},
$$
where $g_{j}(x)=p_{1}(x)\cdots\widehat{p_{j}(x)}\cdots p_{k}(x)$ and $\ \widehat{ }$ refers to dropping the corresponding factor.  
\end{thm}

Employing the extended cover-up method we will now demonstrate our $q$-partial fraction approach. In fact, for our purpose we just need to perform eval on two specific polynomials, namely $\Psi_{n}(x)^{r}$ and $(1-x)^{k}$, where 
$$
\Psi_{n}(x)=\frac{1-x^{n}}{1-x}=x^{n-1}+x^{n-2}+\cdots+x+1.
$$
As explored in Section \ref{sec_psi}, $\Psi_{n}(x)$ fruitfully connects on one hand with the algebraic object eval and on the other with a number theoretic object, the so-called degenerate Bernoulli numbers. 

To explain our methodology we illustrate our procedure for $q$-partial fractions for the simple case $\{n_{1},n_{2}\}$ relatively prime numbers, which implies $\Psi_{n_{1}}(x)$ and $\Psi_{n_{2}}(x)$ are relatively prime. We, therefore, write the generating function as a product of irreducible factors
$$
\frac{1}{(1-x^{n_{1}})(1-x^{n_{2}})}=\frac{1}{(1-x)^{2}\Psi_{n_{1}}(x)\Psi_{n_{2}}(x)}.
$$  
By a direct application of the extended cover-up method we obtain 
\begin{equation}\label{k2case}
\frac{1}{(1-x)^{2}\Psi_{n_{1}}(x)\Psi_{n_{2}}(x)} = \frac{g_{0}(x)}{(1-x)^{2}}+\frac{g_{1}(x)}{(1-x^{n_{1}})}+\frac{g_{2}(x)}{(1-x^{n_{2}})},
\end{equation}
where we multiplied the second and third terms by $(1-x)$ in the numerator and denominator,  
$$
g_{0}(x)=\textnormal{eval}((\Psi_{n_{1}}(x)\Psi_{n_{2}}(x))^{-1}; (1-x)^{2})=c_{0}+c_{1}(1-x)
$$
and 
$$
g_{j}(x)=(1-x)\textnormal{eval}\left(\frac{1}{(1-x)(1-x^{n_{i}})}; \Psi_{n_{j}}(x)\right),
$$ 
for $i,j\in \{1,2\}$ and $i\neq j$. Denote $g_{j}(x)=\frac{1}{n_{j}}\sum^{n_{j}-1}_{\alpha=0}c_{\alpha}^{(j)}x^{\alpha}$.\\
Thus we have a $q$-partial fraction
\begin{equation}\label{two_case}
\frac{1}{(1-x^{n_{1}})(1-x^{n_{2}})}=\frac{c_{0}}{(1-x)^{2}}+\frac{c_{1}}{(1-x)}+\sum_{j=1}^{2}\frac{\sum^{n_{j}-1}_{\alpha=0}c_{\alpha}^{(j)}x^{\alpha}}{n_{j}(1-x^{n_{j}})}.
\end{equation}
With the $q$-partial fraction and using the power series expansion
\begin{equation}\label{power_series}
\frac{1}{(1-x)^{l}}=\sum^{\infty}_{t=0}\begin{pmatrix}
t+l-1\\
t\\
\end{pmatrix}
x^{t},
\end{equation}
we can easily obtain the denumerant, given by coefficient of $x^{t}$ in (\ref{two_case}),
$$
d(t;n_{1},n_{2}) = 
c_{0}(t+1)+c_{1}+\frac{1}{n_{1}}c^{(1)}_{t\%n_{1}}+\frac{1}{n_{2}}c^{(2)}_{t\%n_{2}}.
$$

\begin{table}[!t]
\begin{center}
\begin{tabular}{||c| c||}
\hline
\hline
$\textnormal{eval}\left(\cdot;\  \cdot\right)$ & Result \\
\hline
\hline
$\textnormal{eval}\left(x^{k};1-x^{m}\right)$ &  $x^{k\%m}$\\
\hline
$\textnormal{eval}\left(x^{k};\Psi_{m}(x)\right)$ &  $\left\{\begin{array}{@{}l@{\thinspace}l}
       x^{k\%m} & \text{   if   } k\%m \neq m-1 \\
       -(x^{m-2}+\cdots+1)  & \text{   if   } k\%m=m-1 \\
     \end{array}\right.$\\
\hline
$\res{\frac{1}{\Psi_{n}(x)}}{\Psi_{m}(x)}$ & $\rem{\sum_{j=0}^{a-1}x^{jn\%m}}{\Psi_{m}(x)},\ an\equiv 1\textnormal{ mod } m$\\
\hline
$\res{\frac{1}{(1-x)^{k}}}{\Psi_{m}(x)}$ & $\frac{(-1)^{k}}{k!}\tilde{\beta}_{k}(m,x)\;$ (See Equation \ref{fkm_rdb})\\
\hline 
$\res{\frac{1}{\Psi_{m}(x)^{r}}}{(1-x)^{k+1}}$ & $\frac{1}{m}\sum_{j=0}^{k}\frac{(-1)^{j}}{j!}\tilde{\beta}^{(r)}_{j}(m)(1-x)^{j}$\\
\hline
$\res{\frac{1}{\Psi_{n}(x)}}{\Psi_{m}(x)^{k}}$ & Use $\textnormal{eval}({\frac{1}{\Psi_{n}}(x)};\Psi_{m}(x))$ and Theorem \ref{pow_k_eval}.\\
\hline
\hline
\end{tabular}
\end{center}
\caption{\label{eval_table}Result of the eval operator for some essential functions. Here $\textnormal{gcd}(m,n)=1$, $ r\geq 1$ and $k\geq 0$. These results are proved in Section \ref{sec_psi} and Section \ref{di_lemma}.}    
\end{table}

Further, taking the $t^{th}$ term in the Finite Fourier series expansion (see \cite{Beck}) of the terms $g_{j}(x)/(1-x^{n_{j}})$ we obtain the decomposition
$$
d(t; n_{1},n_{2}) = 
\underbrace{c_{0}(t+1)+c_{1}}_{\textnormal{Polynomial Part}}\ +\ \underbrace{\frac{1}{n_{1}}\sum^{n_{1}-1}_{r=1}g_{1}(\xi_{1}^{r})\xi_{1}^{-rt}+\frac{1}{n_{2}}\sum^{n_{2}-1}_{r=1}g_{2}(\xi_{2}^{r})\xi_{2}^{-rt}}_{\textnormal{Periodic Part}},
$$
where $\xi_{j}=e^{2\pi i/n_{j}}$. Note that the term for $\xi_{j}^{0}=1$ vanishes in the periodic part of the above equation as $(1-x)$ is a factor of $g_{j}(x)$. By a direct extension of the discussion above we obtain the main result Theorem \ref{main_qpf}. The periodic terms are the Sylvester waves corresponding to the roots of unity $\xi_{1}$ and $\xi_{2}$. These terms can be expressed in a pleasant Fourier-Dedekind sum form - thanks to the eval operator! 

For the sake of showing how periodic terms are expressed as Fourier-Dedekind sums let's consider a slightly more general case  $\textbf{A}=(1,\cdots,1,n_{1},\cdots, n_{k})$ with $m$ ones, $n_{1},\cdots, n_{k}$ pairwise relatively prime numbers, $n_{j}\geq 2$ for $1\leq j \leq k$ and denoting $g_{j}(x)=\sum^{n_{j}-1}_{i=0}c_{i}^{(j)}x^{i}$ in (\ref{gj_term_main_qpf}) we obtain 
\begin{eqnarray}
\nonumber d(t;\textbf{A})&=&\sum^{m+k-1}_{j=0}c_{j}
\begin{pmatrix}
t+m+k-j-1\\
t\\
\end{pmatrix}
+\sum^{k}_{j=1}\frac{1}{n_{j}}c^{(j)}_{t\%n_{j}}\\
\nonumber &=& \underbrace{\sum^{m+k-1}_{j=0}c_{j}
\begin{pmatrix}
t+m+k-j-1\\
t\\
\end{pmatrix}}_{\textnormal{Polynomial Part}} + \underbrace{\sum_{j=1}^{k}\frac{1}{n_{j}}\sum^{n_{j}-1}_{r=1}g_{j}(\xi_{j}^{r})\xi_{j}^{-rt}}_{\textnormal{Periodic Part}},
\end{eqnarray}
where $\xi_{j}=e^{2\pi i/n_{j}}$, which is a quasi-polynomial with periodicities in the constant term.\\
By Lemma \ref{eval_eval}, if $\textnormal{gcd}(g(x),\Psi_{m}(x))=1$ we have a neat substitution  
\begin{equation}\label{neat_sub}
\left.\textnormal{eval}\left(\frac{f(x)}{g(x)}; \Psi_{m}(x)\right)\right\rvert_{x=\xi_{m}}=\frac{f(\xi_{m})}{g(\xi_{m})}\quad \textnormal{ for }\quad \xi_{m}=e^{2\pi i/m}.
\end{equation}
Therefore, applying (\ref{neat_sub}) on the equation (\ref{gj_term_main_qpf}) we have the $j^{th}$ periodic term 
$$
\frac{1}{n_{j}}\sum^{n_{j}-1}_{r=1}g_{j}(\xi_{j}^{r})\xi_{j}^{-rt}=\frac{1}{n_{j}}\sum_{r=1}^{n_{j}-1} \frac{p(\xi_{j}^{r})\xi_{j}^{-rt}}{(1-\xi_{j}^{r})^{m}(1-\xi_{j}^{rn_{1}})\cdots \widehat{(1-\xi_{j}^{rn_{j}})}\cdots (1-\xi_{j}^{rn_{k}})},
$$
where $\widehat{ }$ refers to dropping the corresponding term. This formula motivates us to generalize the Fourier-Dedekind sum as given in Definition \ref{gen_FDS}. In the notation of the generalized Fourier-Dedekind sum we have  
\begin{equation}\label{denumerant_fds}
d(t;\textbf{A})=\sum^{m+k-1}_{j=0}c_{j}
\begin{pmatrix}
t+m+k-j-1\\
t\\
\end{pmatrix} +\sum_{j=1}^{k}S^{(m,1, p(x))}_{-t}(n_{1},\cdots,\widehat{n_{j}},\cdots,n_{k}; n_{j}),    
\end{equation}
where again $\widehat{ }$ refers to dropping the term. Using the equation (\ref{denumerant_fds}) we prove the Reciprocity Theorem \ref{k_dim_recip_thm}.

\section{Our Algebraic Approach for Partial Fractions: Extended Cover-Up Method}\label{sec_PFs}

In this section we provide a rigorous treatment of the eval operator and, subsequently, prove the extended cover-up method given in Theorem \ref{cover_up}. In order to achieve this we briefly review certain notation and terminology from commutative algebra, and formalize the polynomial remainder operator. 

An element $p(x)\in \mathbb{Q}[x]$ is a polynomial of the form 
$$
p(x)=\alpha_{m}x^{m}+\cdots+\alpha_{0},\ \  \alpha_{i}\in \mathbb{Q}, \textnormal{ for } i=0,\cdots, m.
$$ 
It is well known that $\mathbb{Q}[x]$ is a principal ideal domain. We denote the ideal generated by a polynomial with the same letter but in a italic font, for instance the prime ideal generated by $p(x)$ is denoted $\mathfrak{p}$. Given a non-empty multiplicative set $S$ (i.e., if $a,b\in S$ then $ab\in S$, and $0\notin S$) of $\mathbb{Q}[x]$, the ring of fractions of $\mathbb{Q}[x]$ with respect to $S$ will be denoted by $S^{-1}\mathbb{Q}[x]$ and is defined by 
$$
S^{-1}\mathbb{Q}[x]=\left\{\frac{f(x)}{g(x)}\in \mathbb{Q}(x):f(x)\in \mathbb{Q}[x] \textnormal{ and } g(x)\in S\right\}.
$$  
Let $a(x)$ factorize as 
\begin{equation}\label{ax}
a(x)=p_{1}(x)^{e_1}\cdots p_{k}(x)^{e_k},
\end{equation}
where $p_{1}(x),\cdots, p_{k}(x)$ are distinct irreducible polynomials. We localize the ring $\mathbb{Q}[x]$ about the multiplicative set 
$
S_{\mathfrak{a}}=\mathbb{Q}[x]-(\mathfrak{p}_{1}\cup\cdots \cup\mathfrak{p}_{k}).
$ 
The localization $S^{-1}_{\mathfrak{a}}\mathbb{Q}[x]$ of $\mathbb{Q}[x]$ is \textit{the set of all rational functions with denominators not sharing any non-trivial common factor with $a(x)$}. It can be easily shown that $S^{-1}_{\mathfrak{a}}\mathbb{Q}[x]=\cap_{i=1}^{k}S^{-1}_{\mathfrak{p}_{i}}\mathbb{Q}[x]$. 

For the purpose of defining the eval operator we will express the remainder operator of $f(x)$ divided by $p(x)$ (denoted $\rem{f(x)}{p(x)}$) as a composition of two functions: the function $\pi_{\mathfrak{p}}:\mathbb{Q}[x]\rightarrow \mathbb{Q}[x]/\mathfrak{p}$ defined 
$$
\pi_{\mathfrak{p}}(f(x))=f(x)+\mathfrak{p},
$$
and $\eta:\mathbb{Q}[x]/\mathfrak{p}\rightarrow \mathbb{Q}[x]$ defined
$$
\eta(A)=
\begin{cases} 
0 & \textnormal{ if } 0\in A\\
\textnormal{smallest degree polynomial in } A & \textnormal{ if } 0\notin A. 
\end{cases}
$$
The following lemma can be easily shown. 

\begin{lem}\label{rem_lem} The function $\eta\circ \pi_{\mathfrak{p}}$ is a homomorphism with polynomial addition and the following equality holds: 
$(\rem{f(x)}{p(x)})=\eta\circ \pi_{\mathfrak{p}}(f(x)).$  
\end{lem}

In this work, for the purpose of derivations, a crucial role is played by the substitution rules. It is interesting (and also unfortunate) to notice that not enough emphasis was given in the literature about the substitution rule of the remainder operator. Indeed, a generalization of substitution rule to higher degree polynomials follows in similar lines to that of the Remainder Theorem \cite[Theorem 2.7]{laudano}. 

\begin{thm}[Generalized Remainder Theorem \cite{laudano}]
Let $f(x),g(x)\in \mathbb{Q}[x]$ with $g(x)=x^{k}+a_{k-1}x^{k-1}+\cdots+a_{1}x+a_{0}$ and suppose $k<\textnormal{deg}(f)$. Then, the remainder of $f(x)$ modulo $g(x)$ can be found by substituting the polynomial $s(x)=-a_{k-1}x^{k-1}-\cdots-a_{1}x-a_{0}$ for $x^{k}$ in $f(x)$, and iterating this substitution in the polynomial thus obtained until its degree becomes less than $k$.
\end{thm}

This substitution rule could also be justified by Lemma \ref{rem_lem} in the sense that repeated substitutions retain the resulting element in the coset, given by $\pi_{\mathfrak{g}}$, and the substitution process terminates with the smallest degree polynomial, picked by $\eta$.    

Equipped with the above substitution rule and using some tricks for manipulation of polynomials one can obtain efficient algorithms to determine the remainder.  

\begin{exmp}
The remainder of $x^{101}+5x^{31}$ by $\Psi_{3}(x)=x^{2}+x+1$ can be obtained by first observing that the polynomial $x^{2}+x+1$ divides $x^{3}-1$ so before substituting $x^{2}$ with $-x-1$ we can first substitute $x^{3}$ with $1$ leading us to 
\begin{eqnarray}
\nonumber \rem{(x^{101}+5x^{31})}{(x^{2}+x+1)}&=& \rem{(x^{101\%3}+5x^{31\%3})}{(x^{2}+x+1)}\\
\nonumber &=& \rem{(x^2+5x)}{(x^{2}+x+1)}\\
\nonumber &=& (-x-1)+5x=4x-1.
\end{eqnarray}
Here we used the fact that the substitution of $x^3$ by $1$ in the factors of $x^k$ we obtain $x^{k\%3}$, where $\%$ is the remainder of integers. This observation leads us to the substitution rule given in Lemma \ref{cor_phi} for the polynomials $\Psi_{m}(x)$.
\end{exmp}

Let $r(x),s(x)$ and $a(x)$ be in the ring of polynomials $\mathbb{Q}[x]$ such that $s(x)$ and $a(x)$ are relatively prime. By B\'{e}zout's identity there exist two polynomials $\alpha(x),\beta(x)\in \mathbb{Q}[x]$ such that 
\begin{equation}\label{a_alpha}
\alpha(x)s(x)+\beta(x)a(x)=1.
\end{equation}

\begin{defn}\label{eval}
Given non-constant polynomials $r(x), s(x), a(x)\in \mathbb{Q}[x]$ and $a(x),s(x)$ satisfy (\ref{a_alpha}), we define the evaluation of the fraction $\frac{r(x)}{s(x)}$ modulo $a(x)$ from $S^{-1}_{\mathfrak{a}}\mathbb{Q}[x]$ to $\mathbb{Q}[x]$ as
$$
\res{\frac{r(x)}{s(x)}}{a(x)} = \rem{(\alpha(x)r(x))}{a(x)},
$$ 
where $\alpha(x)$ is the inverse of $s(x)$ modulo $a(x)$, i.e. $\alpha(x) s(x)=1 \textnormal{ mod }a(x)$. 
\end{defn}
 
The eval operator is well-defined due the following easy to prove isomorphism. 
\begin{lem}\label{isomorphism}
The rings $S_{\mathfrak{a}}^{-1}\mathbb{Q}[x]/S_{\mathfrak{a}}^{-1}\mathfrak{a}$ and $\mathbb{Q}[x]/\mathfrak{a}$ are isomorphic. 
\end{lem}

Suppose $h:S_{\mathfrak{a}}^{-1}\mathbb{Q}[x]/S_{\mathfrak{a}}^{-1}\mathfrak{a}\rightarrow \mathbb{Q}[x]/\mathfrak{a}$ is the isomorphism and using the functions $\pi_{\mathfrak{p}}$ and $\eta$, we can now state the eval function as a sequence of compositions
$$
 S_{\mathfrak{a}}^{-1}\mathbb{Q}[x] \xrightarrow[]{\pi_{S_{\mathfrak{a}}^{-1}\mathfrak{a}}} S_{\mathfrak{a}}^{-1}\mathbb{Q}[x]/S_{\mathfrak{a}}^{-1}\mathfrak{a}\xrightarrow[]{h}\mathbb{Q}[x]/\mathfrak{a}\xrightarrow[]{\eta}\mathbb{Q}[x].
$$
By (\ref{a_alpha}) and observing that $$\frac{-\alpha(x)r(x)\beta(x)a(x)}{1-\beta(x)a(x)}\in S_{\mathfrak{a}}^{-1}\mathfrak{a}$$ we have 
\begin{eqnarray}
\nonumber \res{\frac{r(x)}{s(x)}}{a(x)} &=& \eta\circ h\circ \pi_{S_{\mathfrak{a}}^{-1}\mathfrak{a}} \left(\frac{r(x)}{s(x)}\right) =\eta\circ h\left(\frac{r(x)}{s(x)}+S_{\mathfrak{a}}^{-1}\mathfrak{a}\right)\\
\nonumber &=& \eta\circ h\left(\frac{\alpha(x)r(x)}{1-\beta(x)a(x)}+\frac{-\alpha(x)r(x)\beta(x)a(x)}{1-\beta(x)a(x)}+S_{\mathfrak{a}}^{-1}\mathfrak{a}\right). 
\end{eqnarray}
Upon simplification by using the definitions of $h$ and $\eta$, and Lemma \ref{rem_lem} we get 
\begin{eqnarray}
\nonumber \res{\frac{r(x)}{s(x)}}{a(x)} &=& \eta\circ h\left(\alpha(x)r(x)+S_{\mathfrak{a}}^{-1}\mathfrak{a}\right) \\
\nonumber &=& \eta\left(\alpha(x)r(x)+\mathfrak{a}\right) = \rem{\alpha(x)r(x)}{a(x)}.
\end{eqnarray}
Therefore, $\textnormal{eval}(\cdot,a(x))$ operator can be expressed as the homomorphism $\eta\circ h \circ \pi_{S^{-1}_{\mathfrak{a}}\mathfrak{a}}$. Hence, eval is a well defined operator which takes in a rational polynomial and gives us a polynomial. 

The following properties can also be proved easily from the definitions. 

\begin{lem} \label{lem_res}
Given $a(x), r_{i}(x), s_{i}(x)\in \mathbb{Q}[x]$ and $\textnormal{gcd}(s_{i}(x),a(x))=1$ for $i=0,1$. The following are some properties of $\textnormal{eval}$ function:
\begin{enumerate}
\item $\res{r_{0}(x)}{a(x)}=\rem{r_{0}(x)}{a(x)}$
\item $\res{\frac{r_{0}(x)r_{1}(x)}{s_{0}(x)s_{1}(x)}}{a(x)}=\rem{\left\{\res{\frac{r_{0}(x)}{s_{0}(x)}}{a(x)}\res{\\\frac{r_{1}(x)}{s_{1}(x)}}{a(x)}\right\}}{a(x)}$
\item (Substitution Rule) $\res{\frac{r_{0}(x) - p_{0}(x)a(x)}{s_{0}(x)- q_{0}(x) a(x)}}{a(x)}=\res{\frac{r_{0}(x)}{s_{0}(x)}}{a(x)}$ \textnormal{for some} $p_{0}(x),q_{0}(x)\in \mathbb{Q}[x]$; or stated in words, one can successively substitute the occurrences of $a(x)$ by $0$ in both numerator and denominator of the rational polynomial given in the first argument of eval. 
\end{enumerate}
\end{lem}


\subsection{Partial Fraction Decomposition}

Given $p_{1}(x),\cdots,p_{n}(x)\in \mathbb{Q}[x]$ non-trivial polynomials that are pairwise relatively prime i.e., $\textnormal{gcd}(p_{i}(x),p_{j}(x))=1$ for $i\neq j$. Algebraically speaking, a partial fraction decomposition of the rational function of the form 
\begin{equation}\label{pf_basic}
\frac{1}{p_{1}(x)\cdots p_{n}(x)}=\frac{a_{1}(x)}{p_{1}(x)}+\cdots+\frac{a_{n}(x)}{p_{n}(x)}
\end{equation}
can be viewed as the B\'{e}zout's identity 
$$
a_{1}(x)g_{1}(x)+\cdots+a_{n}(x)g_{n}(x)=1,
$$
where 
\begin{equation}\label{gj}
g_{j}(x)=\prod_{i=1,  i\neq j}^{n}p_{i}(x),\textnormal{ for each } 1\leq j\leq n
\end{equation}
which are coprime, $\textnormal{gcd}(g_{1}(x),\cdots, g_{n}(x))=1$. 

\begin{lem} \label{lem_rem}
Suppose the equation, for $g_{j}(x)$ in (\ref{gj}),  
\begin{equation}\label{lem_bez1}
b_{1}(x)g_{1}(x)+\cdots+b_{n}(x)g_{n}(x)=f(x)
\end{equation}
holds for $f(x)\in \mathbb{Q}[x]$  and $\textnormal{deg }(f)<\textnormal{deg }(p_{1}\cdots p_{n})$ . Then, the following equation also holds 
\begin{equation}\label{bez_0}
(\rem{b_{1}(x)}{p_{1}(x)})g_{1}(x)+\cdots+(\rem{b_{n}(x)}{p_{n}(x)})g_{n}(x)=f(x).
\end{equation}
\end{lem} 
\begin{proof}[Proof.] To prove (\ref{bez_0}) it suffices to show that the quotients of $b_{i}(x)$ when divided by $p_{i}(x)$ add up to zero. Suppose we have
$
b_{i}(x)=p_{i}(x)q_{i}(x)+r_{i}(x)
$
where $\textnormal{deg}(r_{i})<\textnormal{deg}(p_{i})$, for $i=1,\cdots,n$. Dividing both sides of the equation (\ref{lem_bez1}) by $p_{1}(x)\cdots p_{n}(x)$ we get 
$$
\frac{f(x)}{p_{1}(x)\cdots p_{n}(x)}=\frac{b_{1}(x)}{p_{1}(x)}+\cdots+\frac{b_{n}(x)}{p_{n}(x)}=\left(q_{1}(x)+\frac{r_{1}(x)}{p_{1}(x)}\right)\cdots+\left(q_{n}(x)+\frac{r_{n}(x)}{p_{n}(x)}\right).
$$
After rearranging the terms we have a difference of two proper rational polynomials equal to a polynomial  
$$
\left(\frac{f(x)}{p_{1}(x)\cdots p_{n}(x)}\right)-\left(\frac{r_{1}(x)}{p_{1}(x)}+\cdots+\frac{r_{n}(x)}{p_{n}(x)}\right)=q_{1}(x)+\cdots+q_{n}(x),
$$
which is possible only if $q_{1}(x)+\cdots+q_{n}(x)=0$. Hence the result is proved.   
\end{proof}

The advantage of having the above lemma is in simplifying the partial fraction expansion for complicated numerators. We now prove our fundamental lemma which says that eval operator converts certain equations to B\'{e}zout's identities. 

\begin{lem}\label{lem_bez2}[Fundamental Lemma]
Suppose $g_{j}(x)$ satisfy (\ref{gj}) and 
\begin{equation}\label{linear_eqn}
a_{1}(x)g_{1}(x)+\cdots+a_{n}(x)g_{n}(x)=f(x)
\end{equation}
holds for $\textnormal{deg }(f)<\textnormal{deg }(p_{1}\cdots p_{n})$ and $\textnormal{gcd}(f(x),p_{i}(x))=1$ for all $i=1,2,\cdots,n$. Then, the following equation also holds:
\begin{equation}\label{bez_1}
\res{\frac{a_{1}(x)}{f(x)}}{p_{1}(x)}g_{1}(x)+\cdots+\res{\frac{a_{n}(x)}{f(x)}}{p_{n}(x)}g_{n}(x)=1.
\end{equation}
\end{lem} 
\begin{proof}[Proof.]
As $\textnormal{gcd}(f(x),p_{i}(x))=1$ for $1\leq i \leq n$ there exist $\alpha_{i}(x),\beta_{i}(x)\in \mathbb{Q}[x]$ such that
$$
\alpha_{i}(x)p_{i}(x)+\beta_{i}(x)f(x)=1.
$$ 
By the definition of eval function it suffices to prove the equation (\ref{bez_1}) with the factors $\res{\frac{a_{i}(x)}{f(x)}}{p_{i}(x)}$ replaced with $\rem{a_{i}(x)\beta_{i}(x)}{p_{i}(x)}$. Dividing both sides of the equation (\ref{linear_eqn}) by $f(x)$ and $\prod_{i=1}^{n}p_{i}(x)$ we have 
$$
\frac{1}{p_{1}(x)\cdots p_{n}(x)}=\frac{a_{1}(x)}{f(x)p_{1}(x)}+\cdots+\frac{a_{n}(x)}{f(x)p_{n}(x)}.
$$  
By multiplying in the numerator with the factor $(\alpha_{i}(x)p_{i}(x)+\beta_{i}(x)f(x))$ (which essentially equals $1$) in the $i^{th}$ term we have
$$
\frac{1}{p_{1}(x)\cdots p_{n}(x)}=\sum^{n}_{i=1}\frac{a_{i}(x)(\alpha_{i}(x)p_{i}(x)+\beta_{i}(x)f(x))}{f(x)p_{i}(x)}.
$$ 
Upon simplifying and rearranging terms we have
$$
\frac{a_{1}(x)\alpha_{1}(x)+\cdots+a_{n}(x)\alpha_{n}(x)}{f(x)}=\frac{1}{p_{1}(x)\cdots p_{n}(x)}-\sum_{i=1}^{n}\frac{a_{i}(x)\beta_{i}(x)}{p_{i}(x)},
$$
i.e. left hand side has a denominator $f(x)$ and the denominator of the right hand side is coprime to $f(x)$. This is possible only if $a_{1}(x)\alpha_{1}(x)+\cdots+a_{n}(x)\alpha_{n}(x)=0$. Therefore, the right hand side of the above equation is zero. Hence the lemma is established.    
\end{proof}

We are now ready to prove the extended cover-up method.   

\begin{proof}[Proof of Theorem \ref{cover_up}.]
By Lemma \ref{lem_rem}, without loss of generality, we can assume $f(x)=1$. The polynomials $g_{j}(x)=\prod^{n}_{i=1,i\neq j}p_{i}(x)$ for $j=1,\cdots,n$ are mutually coprime. By B\'{e}zout's identity there exist $a_{j}(x)\in \mathbb{Q}[x]$ for $1\leq j\leq n$ such that 
$
a_{1}(x)g_{1}(x)+\cdots+a_{n}(x)g_{n}(x)=1.
$
It can be easily seen that none of the polynomials $a_{j}(x)$ are zero. So, we consider the identity 
$$
\left(\frac{a_{1}(x)}{a_{1}(x)g_{1}(x)+\cdots+a_{n}(x)g_{n}(x)}\right)g_{1}(x)+\cdots+\left(\frac{a_{n}(x)}{a_{1}(x)g_{1}(x)+\cdots+a_{n}(x)g_{n}(x)}\right)g_{n}(x)=1.
$$
Then by Lemma \ref{lem_bez2} we can write the above equation as  
$$
\sum^{n}_{i=1}\res{\frac{a_{i}(x)}{a_{1}(x)g_{1}(x)+\cdots+a_{n}(x)g_{n}(x)}}{p_{i}(x)}g_{i}(x)=1.
$$ 
By substitution rule Lemma \ref{lem_res}(3) we have 
$$
\sum^{n}_{i=1}\res{\frac{a_{i}(x)\mod p_{i}(x)}{a_{1}(x)g_{1}(x)+\cdots+a_{n}(x)g_{n}(x)\mod p_{i}(x)}}{p_{i}(x)}g_{i}(x)=1.
$$ 
As $g_{j}(x)\mod p_{i}(x)=0$ if $i\neq j$ and $a_{j}(x)$ non-zero for all $j=1,\cdots,n$
$$
\res{\frac{1}{g_{1}(x)}}{p_{1}(x)}g_{1}(x)+\cdots+\res{\frac{1}{g_{n}(x)}}{p_{n}(x)}g_{n}(x)=1
$$
Hence the result is proved.  
\end{proof}

Our method for partial fraction decomposition can be considered as an alternative to the undetermined coefficient method, usually taught at an early stage of higher mathematics. To illustrate the efficacy of our method we provide an example.

\begin{exmp}
Using the extended cover-up method we determine the partial fraction of $\frac{1}{(x^2+a^2)(x^3+b^3)}$, for arbitrary but fixed $a,b>0$. Let the partial fraction be given by
$$
\frac{1}{(x^2+a^2)(x^3+b^3)} = \frac{a_{1}(x)}{(x^2+a^2)}+\frac{a_{2}(x)}{(x^3+b^3)}
$$
we now evaluate for $a_{1}(x)$ and $a_{2}(x)$. 
\begin{eqnarray}
\nonumber a_{1}(x)&=&\res{\frac{1}{x^{3}+b^{3}}}{x^{2}+a^{2}} = \res{\frac{1}{x^{3}+b^{3}}\times \frac{x^{3}-b^{3}}{x^{3}-b^{3}}}{x^{2}+a^{2}}\\
\nonumber &=& \res{\frac{x^{3}-b^{3}}{x^{6}-b^{6}}}{x^{2}+a^{2}}= \frac{a^{2}x+b^{3}}{a^{6}+b^{6}},
\end{eqnarray}
where we used the substitution rule given in Lemma \ref{lem_res}(3). Similarly, we have  
\begin{eqnarray}
\nonumber a_{2}(x)&=&\res{\frac{1}{x^{2}+a^{2}}\times \frac{x^{4}-a^{2}x^{2}+a^{4}}{x^{4}-a^{2}x^{2}+a^{4}}}{x^{3}+b^{3}}= \frac{-a^{2}x^{2}-b^{3}x-a^{4}}{b^{6}+a^{6}}
\end{eqnarray}
Hence, the partial fraction is 
$$
\frac{1}{(x^2+a^2)(x^3+b^3)}=\frac{1}{a^{6}+b^{6}}\left(\frac{a^{2}x+b^{3}}{x^{2}+a^{2}}+\frac{-a^{2}x^{2}-b^{3}x-a^{4}}{x^{3}+b^{3}}\right)
$$
One can factorize $(x^{3}+b^{3})$ and follow the above method to perform further decomposition. Note that by the undetermined coefficient method this decomposition would have required solving simultaneous linear equations with five unknowns.  
\end{exmp}

\subsection{Partial Fraction for Powers}
Given a partial fraction decomposition corresponding to 
\begin{equation}\label{basic_pf}
a_{1}(x)p_{1}(x)+a_{2}(x)p_{2}(x)=1,
\end{equation}
one can obtain a direct formula for partial fractions by the extended cover-up of $1/(p_{1}(x)^{k_{1}}p_{2}(x)^{k_{2}}).$ 

Raising both sides of (\ref{basic_pf}) by $n$ and rearranging the terms we get
\begin{equation}\label{n_basic_pf}
a_{1}(x)^{n}p_{1}(x)^{n}+a_{2}(x)^{n}p_{2}(x)^{n}=1-p_{1}(x)q(x),
\end{equation}
where $q(x)=\sum_{k=1}^{n-1}\begin{pmatrix}n \\ k\end{pmatrix}a_{1}(x)^{k}a_{2}(x)^{n-k}p_{2}(x)^{n-k}p_{1}(x)^{k-1}$. Further, dividing both sides of (\ref{n_basic_pf}) by $1-p_{1}(x)q(x)$ we get 
$$
\left(\frac{a_{1}(x)^{n}}{1-p_{1}(x)q(x)}\right)p_{1}(x)^{n}+\left(\frac{a_{2}(x)^{n}}{1-p_{1}(x)q(x)}\right)p_{2}(x)^{n}=1.
$$
Thus, in view of the extended cover-up method and the substitution rule Lemma \ref{lem_res}(3), we have 
\begin{eqnarray}
\nonumber \textnormal{eval}\left(\frac{1}{p_{2}(x)^{n}}; p_{1}(x)^{n}\right)&=&\textnormal{eval}\left(\frac{a_{2}(x)^{n}}{1-p_{1}(x)q(x)};p_{1}(x)^{n}\right)\\
\nonumber &=& \textnormal{eval}\left(a_{2}(x)^{n}\sum_{k=0}^{n-1}p_{1}(x)^{k}q(x)^{k}; p_{1}(x)^{n}\right)\\
\nonumber &=& \left(a_{2}(x)^{n}\sum_{k=0}^{n-1}p_{1}(x)^{k}q(x)^{k}\right) \textnormal{ rem } p_{1}(x)^{n}.
\end{eqnarray}

\begin{thm}\label{pow_k_eval}
Suppose $a_{1}(x)p_{1}(x)+a_{2}(x)p_{2}(x)=1$. Then for a given positive integer $n$ we have 
$$
\textnormal{eval}\left(\frac{1}{p_{2}(x)}; p_{1}(x)^{n}\right)=\left(p_{2}(x)^{n-1}a_{2}(x)^{n}\sum_{k=0}^{n-1}p_{1}(x)^{k}q(x)^{k}\right) \textnormal{ rem } p_{1}(x)^{n},
$$
where 
$
q(x)=\sum_{k=1}^{n-1}\begin{pmatrix}n \\ k\end{pmatrix}a_{1}(x)^{k}a_{2}(x)^{n-k}p_{2}(x)^{n-k}p_{1}(x)^{k-1}.
$ 
\end{thm}

Using Theorem \ref{pow_k_eval}, Lemma \ref{lem_res} and the extended cover-up method we can derive a partial fraction formula for the higher power case.

\section{The polynomial $\Psi_{m}(x)$ }\label{sec_psi}

The polynomial $\Psi_{m}(x)$ plays a key role in the $q$-partial fractions owing to $$1-x^{m}=(1-x)\Psi_{m}(x).$$ Moreover, we observe that $\Psi_{m}(x)$ shares a pleasant relationship with the eval operator and also the special classes of numbers, the so-called degenerate Bernoulli numbers.  

The (usual) Bernoulli number $B_{k}$ may be defined by the generating function 
$$
\frac{t}{e^{t}-1}=\sum^{\infty}_{k=0}B_{k}\frac{t^{k}}{k!}.
$$
Since $(1+\lambda t)^{1/\lambda}\rightarrow e^{t}$ as $\lambda \rightarrow 0$, Carlitz \cite{degenerate_bernoulli} defined the degenerate Bernoulli number as 
\begin{equation}\label{carlitz_eqn}
\frac{t}{(1+\lambda t)^{1/\lambda}-1}=\sum^{\infty}_{k=0}\beta_{k}(\lambda)\frac{t^{k}}{k!}.
\end{equation}
By setting $\lambda=1/m$ and $x=1+\lambda t$ in (\ref{carlitz_eqn}) we obtain 
\begin{equation}\label{young_eqn}
\frac{m(1-x)}{1-x^{m}}=\frac{m}{\Psi_{m}(x)}=\sum_{k=0}^{\infty}(-1)^{k}\frac{\tilde{\beta}_{k}(m)}{k!}(1-x)^{k},
\end{equation}
where $\tilde{\beta}_{k}(m)=m^{k}\beta_{k}(1/m)$ where $\beta_{k}(m)$ is the degenerate Bernoulli number. Thus $\tilde{\beta}_{k}(m)$ is the reciprocal polynomial of $\beta_{k}(m)$. See \cite{young} for some properties of $\tilde{\beta}_{k}(m)$.

The reciprocal degenerate Bernoulli of order $r$ is given by \cite{young,Zhang}  
\begin{equation}\label{r_RDB}
\left(\frac{m}{\Psi_{m}(x)}\right)^{r}=\sum_{k=0}^{\infty}(-1)^{k}\frac{\tilde{\beta}^{(r)}_{k}(m)}{k!}(1-x)^{k}.
\end{equation}
Raising both sides of (\ref{young_eqn}) by a power $r$ and expanding we have 
$$
\tilde{\beta}_{k}^{(r)}(m)=\sum_{k_{1}+\cdots+k_{r}=k}\frac{k!}{k_{1}!\cdots k_{r}!}\tilde{\beta}_{k_{1}}(m)\cdots \tilde{\beta}_{k_{r}}(m).
$$
Obviously, $\tilde{\beta}^{(1)}_{k}(m)=\tilde{\beta}_{k}(m)$. A sum of products formula for $\tilde{\beta}_{k}(m)$ with a fixed $m$ is given in \cite{Zhang}.

\subsection{Substitution Rules of $\Psi_{m}(x)$}
We will now explore the relation between $\Psi_{m}(x)$ and the eval function. 
\begin{lem}[Substitution Rule 1]\label{cor_phi}
\begin{equation}\label{rem_psi}
\textnormal{eval}(x^{j};\Psi_{m}(x))=x^{j} \textnormal{ rem } \Psi_{m}(x)= \left\{\begin{array}{@{}l@{\thinspace}l}
       x^{j\%m} & \textnormal{   if   } j\%m \neq m-1 \\
       -\sum^{m-2}_{i=0}x^{i}  & \textnormal{   if   } j\%m=m-1. \\
     \end{array}\right.
\end{equation}
where $\%$ operator is the remainder in the ring of integers. 
\end{lem}
\begin{proof}[Proof.]
Suppose $j=md+r$, for $0\leq r<m$. Then, the result follows from 
$
x^{md+r}=((x^{m}-1)+1)^{d}x^{r}=q(x)(x^{m}-1)+x^{r}=q(x)(x-1)\Psi_{m}(x)+x^{r}.
$ 
\end{proof}

The $\textnormal{eval}(\cdot,\Psi_{m}(x)^{r})$ function behaves in a satisfying  manner under the action of Fourier Transform. This can be seen by another substitution rule. 

\begin{lem}[Substitution Rule 2]\label{eval_eval}
Let $f(x),g(x)\in \mathbb{Q}[x]$ and $\textnormal{gcd}(g(x),\Psi_{b}(x))=1$. Suppose
$$
h(x)=\res{\frac{f(x)}{g(x)}}{\Psi_{b}(x)^{r}} \textnormal{  for some positive integers } r \textnormal{ and } b. 
$$
Then, we have $h(\xi_{b})=f(\xi_{b})/g(\xi_{b})$ for $\xi_{b}= e^{2\pi i/b}$.
\end{lem}
\begin{proof}[Proof.]
As $\textnormal{gcd}(g(x),\Psi_{b}(x))=1$, by B\'{e}zout's identity, there exist $\alpha(x),\beta(x)\in \mathbb{Q}[x]$ such that 
\begin{equation}\label{g_psi}
\alpha(x)g(x)+\beta(x)\Psi_{b}(x)^{r}=1.
\end{equation}
Substituting $\xi_{b}=e^{2\pi i/b}$ both sides of (\ref{g_psi}), we get $\alpha(\xi_{b})g(\xi_{b})=1.$
Also, 
\begin{eqnarray}
\nonumber h(x)=\res{\frac{f(x)}{g(x)}}{\Psi_{b}(x)^{r}}&=& f(x)\alpha(x) \textnormal{ rem } \Psi_{b}(x)^{r}\\
\nonumber &=& f(x)\alpha(x)-q(x)\Psi_{b}(x)^{r},
\end{eqnarray}
for some $q(x)\in \mathbb{Q}[x]$. Hence, we have $h(\xi_{b})=f(\xi_{b})\alpha(\xi_{b})=f(\xi_{b})/g(\xi_{b}).$ 
\end{proof}

\subsection{Fourier Series and $\textnormal{eval}(\cdot, \Psi_{b}(x))$}\label{subsec_dft}

The aim of this section is to give a Fourier series for $h(x)/(1-x^{b})^{r}$. It is well known that when $r=1$ the Fourier series is a finite Fourier series. For an excellent discussion on finite Fourier series see \cite{Beck}.

We first determine the finite Fourier series representation of 
$$
\frac{h(x)}{1-x^{b}},\quad \textnormal{ for }\quad h(x)=\textnormal{eval}\left(\frac{f(x)}{g(x)};\Psi_{b}(x)\right),
$$ 
where $f(x),g(x)\in \mathbb{Q}[x]$ and $g(x)$ relatively prime to $\Psi_{b}(x)$. Let $a(n)$ be a periodic function on $\mathbb{Z}$ with a period $b$. Consider the series $H(x)=\sum^{\infty}_{n=0}a(n)x^{n}$ which has the generating function   
$$
H(x)=\frac{h(x)}{1-x^{b}}\quad \textnormal{  for }\quad h(x)=\sum^{b-1}_{j=0}a(j)x^{j}.
$$
The finite Fourier series expansion is given by 
$$
a(n)=\frac{1}{b}\sum_{j=0}^{b-1}h(\xi_{b}^{j})\xi_{b}^{-nj},\quad \textnormal{ where }\quad \xi_{b}=e^{2\pi i/b}.
$$The above coefficients are, essentially, obtained by applying the inverse DFT on $(h(\xi_{b}^{0}),\cdots, h(\xi_{b}^{b-1}))^{T}$.

By $h(x)=\res{f(x)/g(x)}{\Psi_{b}(x)}$, Lemma \ref{eval_eval} and the finite Fourier series expansion we deduce 
\begin{equation}\label{finite_FS}
H(x)=\frac{h(x)}{1-x^{b}} =\sum^{\infty}_{n=0}a(n)x^{n}=\sum^{\infty}_{n=0}\left(\frac{1}{b}\sum_{j=0}^{b-1}h(\xi_{b}^{j})\xi_{b}^{-nj}\right)x^{n}= \sum_{n=0}^{\infty}\left(\frac{1}{b}\sum^{b-1}_{j=0}\frac{f(\xi_{b}^{j})}{g(\xi_{b}^{j})}\xi_{b}^{-jn}\right)x^{n}.
\end{equation}The following example is a connecting link between the $\textnormal{eval}(\cdot; \Psi_{m}(x))$ and the Fourier-Dedekind sum given in \cite{Beck}.
\begin{exmp}\label{eval_FDS}
Let $n_{1},\cdots, n_{k}$ and $b$ be pairwise relatively prime positive integers and suppose $h(x)=(1-x)\res{\frac{1}{(1-x)(1-x^{n_{1}})\cdots (1-x^{n_{k}})}}{\Psi_{b}(x)}$. Then the $n^{th}$ term in the power series of $h(x)/(1-x^{b})$ is 
$$
a(n)=\frac{1}{b}\sum^{b-1}_{j=1}\frac{\xi_{b}^{-jn}}{(1-\xi_{b}^{jn_{1}})\cdots (1-\xi_{b}^{jn_{k}})}.
$$ 
\end{exmp}

Now, we obtain the Fourier series expansion of
$
h(x)/(1-x^{b})^{r}
$
where $r> 1$ and $\textnormal{deg}(h)<rb$.  In this case, the associated Fourier series is no more a finite Fourier series. Nevertheless, one can obtain a simplification by writing 
\begin{equation}\label{decompose_h}
h(x)=\sum^{r-1}_{j=0}h_{j}(x)(1-x^{b})^{j}, \textnormal{ where } \textnormal{deg}(h_{j})<b,
\end{equation}
and thus we can express 
$$
\frac{h(x)}{(1-x^{b})^{r}}=\sum_{j=0}^{r-1}\frac{1}{(1-x^{b})^{r-j-1}}\frac{h_{j}(x)}{(1-x^{b})}.
$$
Hence, we can obtain the finite Fourier series for each $h_{j}(x)/(1-x^{b})$ by applying (\ref{finite_FS}). Therefore, we strive to obtain an explicit formula for (\ref{decompose_h}) which is a Taylor series like expansion in terms of $(1-x^{b})$. Towards this end we introduce a special differentiation notation  
$$
D_{b}(x^{k})=\left\lfloor\frac{k}{b}\right\rfloor x^{k-b}. 
$$
For $b=1$, $D_{b}=D$, the standard differentiation operation. Nevertheless, due to a deficiency in Leibniz rule, $D_{b}$ cannot be, in general, considered as a differential operator.  

\begin{thm}\label{taylor_Dm}
Suppose $h(x)$ is a polynomial with $\textnormal{deg}(h)< rb$. Then, 
$$
h(x)=\sum^{r-1}_{j=0}\frac{(-1)^{j}h^{(j)}(x)}{j!}(1-x^{b})^{j},
$$
where $h^{(j)}(x)=\textnormal{eval}\left(D_{b}^{j}h(x); 1-x^{b}\right)$. In particular, $h^{(0)}(x)$ is the remainder of $h(x)$ when divided by $1-x^{b}$.
\end{thm}

\begin{proof}[Proof.]
Setting $y=x^{b}$, we can write $h(x)=\sum^{r-1}_{i=0}x^{i}\left(\sum^{(r-1)}_{j=0}\alpha_{ij}y^{j}\right)$. The result follows from the Taylor's theorem about $y=1$ on each term $\sum^{(r-1)}_{j=0}\alpha_{ij}y^{j}$.  
\end{proof}

It is easy to compute $h^{(j)}(x)$ in the above Taylor series like expansion by the following lemma. 

\begin{lem} For $k\geq 0$ and $b>1$, 
$
\textnormal{eval}(x^{k};1-x^{b})=x^{k\%b}.
$
\end{lem}

\begin{proof}[Proof.]
By Lemma \ref{lem_res}, the eval of a polynomial is remainder. Let $k=ab+c$ for $0\leq c<b$. Then we can write 
$$
x^{k}=x^{c}(1-(1-x^{b}))^{a}=  \left(\sum^{a}_{j=1}(-1)^{j}\begin{pmatrix} a \\ j \end{pmatrix}(1-x^{b})^{j-1}\right)(1-x^{b})+x^{c}.
$$ 
Hence, the remainder when $x^{k}$ is divided by $1-x^{b}$ equals $x^{c}=x^{k\% b}$. 
\end{proof}

Furthermore, from the above lemma it is also easy to see that substitution of $x=\xi_{b}$ in $h^{(j)}(x)$ gives us 
$$
h^{(j)}(\xi_{b})=\left.\textnormal{eval}\left(D_{b}^{j}h(x); 1-x^{b}\right)\right\rvert_{x=\xi_{b}}=D^{j}_{b}h(\xi_{b}).
$$
In particular, by the substitution rule Lemma \ref{eval_eval}, we have  $h^{(0)}(\xi_{b})=h(\xi_{b})$.

Now the Fourier series for the $j^{th}$ term corresponding to $h^{(j)}(x)$ can be written as 
\begin{equation}\label{gen_hij}
\frac{(-1)^{j}}{j!}\frac{1}{(1-x^{b})^{r-j-1}}\frac{h^{(j)}(x)}{(1-x^{b})} =\sum^{\infty}_{n=0}\sum_{n'=0}^{n}a_{1}(n')a_{2}(n-n')x^{n},
\end{equation}
where $a_{1}(n)=\frac{(-1)^{j}}{j!}\begin{pmatrix}\lfloor n/b\rfloor-r-j-2 \\ r+j+2\end{pmatrix}$ and $a_{2}(n)=\frac{1}{b}\sum_{k=0}^{b-1}h_{j}(\xi_{b}^{k})\xi_{b}^{-nk}$. In particular, for $j=0$ we have 
\begin{equation}\label{zero_term}
\frac{1}{(1-x^{b})^{r-1}}\frac{h^{(0)}(x)}{(1-x^{b})}=\sum_{n=0}^{\infty}\sum_{n'=0}^{n} \begin{pmatrix} \lfloor n'/b\rfloor-r-2 \\ r+2\end{pmatrix}\left(\frac{1}{b}\sum_{k=0}^{b-1}h(\xi_{b}^{k})\xi_{b}^{-(n-n')k}\right)x^{n}.
\end{equation}
Finally, the Fourier series for $h(x)/(1-x^{b})^{r}$ can be obtained by adding all the terms for $j=0,\cdots,(r-1)$ from (\ref{gen_hij}).

\section{Preparation Lemmas}\label{di_lemma}
In this section, we obtain the eval of the functions $(1-x)^{k}$ and $\Psi_{m}(x)$ through basic partial fractions and prove some preparation lemmas for $q$-partial fractions. We start by performing partial fraction for 
$$
\frac{1}{1-x^{m}}=\frac{1}{(1-x)\Psi_{m}(x)}.
$$ 
Differentiating and multiplying by $x$ both sides of the equation
$$
(1-x)\Psi_{m}(x)=1-x^{m}
$$
we get 
$$
x\Psi'_{m}(x)(1-x)+(-x)\Psi_{m}(x)=-mx^{m}.
$$
Applying the Fundamental Lemma \ref{lem_bez2} we get 
$$
\res{\frac{x\Psi'_{m}(x)}{-mx^{m}}}{\Psi_{m}(x)}(1-x)+\res{\frac{-x}{-mx^{m}}}{(1-x)}\Psi_{m}(x)=1.
$$
Simplifying we get the partial fraction  
\begin{equation}\label{pf_basic}
\frac{1}{(1-x)\Psi_{m}(x)}=\frac{1}{m(1-x)}+\frac{\rem{(-\frac{1}{m}x\Psi'_{m}(x))}{\Psi_{m}(x)}}{\Psi_{m}(x)}.
\end{equation}
Proceeding inductively we can prove the following result.
\begin{lem}\label{pow_k}
For $k\geq 1$ the following partial fraction holds:
\begin{equation}\label{pf_k}
\frac{1}{(1-x)^{k}\Psi_{m}(x)} =\frac{1}{m}\sum_{j=0}^{k-1}\frac{f_{j}^{(m)}(1)}{(1-x)^{k-j}}+\frac{f_{k}^{(m)}(x)}{\Psi_{m}(x)},
\end{equation}
where $f_{j}^{(m)}(x)=\rem{(-\frac{1}{m}x\Psi'_{m}(x))^{j}}{\Psi_{m}(x)}$ and $f_{0}^{(m)}(x)=1$.
\end{lem}
\begin{proof}[Proof.]
We prove by induction on $k$. The equation (\ref{pf_basic}) is the base case $k=1$. Suppose the result holds for $k$.  Multiplying both sides of (\ref{pf_k}) by $\frac{1}{1-x}$we get 
$$
\frac{1}{(1-x)^{k+1}\Psi_{m}(x)} =\frac{\sum_{j=0}^{k-1}f_{j}^{(m)}(1)(1-x)^{j}}{m(1-x)^{k+1}}+\frac{f_{k}^{(m)}(x)}{(1-x)\Psi_{m}(x)}
$$
By performing partial fractions on the second term we have 
\begin{eqnarray}
\nonumber \frac{1}{(1-x)^{k+1}\Psi_{m}(x)} &=& \frac{\sum_{j=0}^{k-1}f_{j}^{(m)}(1)(1-x)^{j}}{m(1-x)^{k+1}}+\frac{f_{k}^{(m)}(1)}{m(1-x)}+\frac{f_{k+1}^{(m)}(x)}{\Psi_{m}(x)}\\
\nonumber &=& \frac{\sum_{j=0}^{k}f_{j}^{(m)}(1)(1-x)^{j}}{m(1-x)^{k+1}}+\frac{f_{k+1}^{(m)}(x)}{\Psi_{m}(x)}.
\end{eqnarray}
Hence the result is proved.   
\end{proof}

Comparing the terms in equation (\ref{young_eqn}) and the terms in equation (\ref{pf_k}) we are motivated to define the reciprocal degenerate Bernoulli polynomial 
\begin{equation}\label{fkm_rdb}
\frac{(-1)^{k}}{k!}\tilde{\beta}_{k}(m,x)=f_{k}^{(m)}(x) = \rem{\left(-\frac{1}{m}x\Psi'_{m}(x)\right)^{k}}{\Psi_{m}(x)}.
\end{equation}
The reciprocal degenerate Bernoulli number is given by $\tilde{\beta}_{k}(m)=\tilde{\beta}_{k}(m,1).$ 

Next, we prove another preparation lemma for $q$-partial fractions. 
\begin{lem}\label{gcd}
Let $\textnormal{gcd}(m,n)=1$ and let $a,b>0$ satisfying $am-bn=1$. Then, 
\begin{equation}\label{n_m}
\res{\frac{1}{\Psi_{m}(x)}}{\Psi_{n}(x)} = \rem{\sum_{j=0}^{a-1}x^{jm\%n}}{\Psi_{n}(x)}.
\end{equation}
\end{lem}
\begin{proof}[Proof.]
Observe that we can write 
$$
\res{\frac{1}{\Psi_{m}(x)}}{\Psi_{n}(x)} = \res{\frac{1-x}{1-x^{m}}}{\Psi_{n}(x)}.
$$
Upon multiplying and dividing by $(x^{(a-1)m}+\cdots+1)$ we have 
\begin{equation}\label{eq_am}
\res{\frac{1-x}{1-x^{m}}}{\Psi_{n}(x)}=\res{\frac{(1-x)(x^{(a-1)m}+\cdots+1)}{1-x^{am}}}{\Psi_{n}(x)}.
\end{equation}
By Lemma \ref{lem_res}(3), we can replace any one of the equivalent polynomials modulo $\Psi_{n}(x)$ at both numerator and the denominator. Moreover, by the Lemma \ref{cor_phi} when we perform modulo $\Psi_{n}(x)$ we can first work modulo $(1-x^{n})$ and replacing $x^{n}$ by $1$. As $am-bn=1$, we have 
$$
1-x^{am}=1-x^{am}x^{-bn}=1-x^{am-bn}=1-x  \mod (1-x^{n}).
$$
Therefore, from (\ref{eq_am}) we can write 
\begin{eqnarray}
\nonumber \res{\frac{1}{\Psi_{m}(x)}}{\Psi_{n}(x)} &=& \res{\frac{(1-x)(x^{(a-1)m}+\cdots+1)}{1-x}}{\Psi_{n}(x)}\\
\nonumber &=& \res{x^{(a-1)m}+\cdots+1}{\Psi_{n}(x)}\\
\nonumber &=& \rem{\sum_{j=0}^{a-1}x^{jm\%n}}{\Psi_{n}(x)}.
\end{eqnarray} 
\end{proof}

The above result can be made precise to provide an exact partial fraction formula. As the proof can be directly obtained with the help of above lemma, we state the result without proof. 

\begin{thm}
Let $m,n$ be two integers that are coprime and, let $m\geq 2$ and $n\geq3$. Consider the smallest possible positive numbers $a,b,\alpha,\beta$ for which we have $am-bn=1$ and $\alpha n- \beta m=1$. If $a>\beta$ then the partial fraction is given by 
$$
\frac{1-x}{(1-x^{m})(1-x^{n})}=\frac{\sum_{j=0}^{\alpha-1}x^{jn\%m}}{1-x^{m}}-\frac{\sum_{j=0}^{\beta-1}x^{(jm+1)\%n}}{1-x^{n}}
$$
else if $a<\beta$ we have the partial fraction 
$$
\frac{1-x}{(1-x^{m})(1-x^{n})}=\frac{\sum_{j=0}^{a-1}x^{jm\%n}}{1-x^{n}}-\frac{ \sum_{j=0}^{b-1}x^{(jn+1)\%m}}{1-x^{m}}.
$$
\end{thm}  
A similar result but addressing a different case is considered using iterated Laurent series in \cite[Proposition 4.3, Lemma 4.5]{Xin}.

\section{$q$-Partial Fractions and Denumerants}\label{subsec_qpf}
Consider the general rational polynomial for $q$-partial fractions  
\begin{equation}\label{general_qpf}
F(x)=\frac{p(x)}{(1-x)^{m}(1-x^{n_{1}})^{r_{1}}\cdots(1-x^{n_{k}})^{r_{k}}},
\end{equation}
where $n_{1},\cdots, n_{k}$ are pairwise relatively primes, $s=r_{1}+\cdots+r_{k}$, $\textnormal{deg }(p)<m+r_{1}n_{1}+\cdots+r_{k}n_{k}$. We determine the $q$-partial fraction of (\ref{general_qpf}) by the factorized expression
\begin{equation}\label{gen_qpf}
 \frac{p(x)}{(1-x)^{m+s}\Psi_{n_{1}}(x)^{r_{1}}\cdots \Psi_{n_{k}}(x)^{r_{k}}}=\frac{h_{0}(x)}{(1-x)^{m+s}} + \sum^{k}_{j=1}\frac{h_{j}(x)}{(1-x^{n_{j}})^{r_{j}}}.
\end{equation}
By the extended cover-up method we have  
$$
h_{0}(x)=\left(p(x)\prod_{j=1}^{k}\textnormal{eval}\left(\frac{1}{\Psi_{n_j}(x)^{r_{j}}}; (1-x)^{m+s}\right)\right) \textnormal{ rem }(1-x)^{m+s}.
$$
By expanding the terms of $h_{0}(x)$ using Table \ref{eval_table} we have
\begin{equation}\label{h0_term}
h_{0}(x)=\sum^{m+s-1}_{j=0}c_{j}(1-x)^{j},
\end{equation}
where, for $0\leq j<m+s$,
$$
c_{j}=\frac{1}{n_{1}^{r_{1}}\cdots n_{k}^{r_{k}}}\sum_{i_{0}+i_{1}+\cdots+i_{k}=j}(-1)^{j}\frac{D^{i_{0}}p(1)}{i_{0}!}\frac{\tilde{\beta}^{(r_{1})}_{i_{1}}(n_{1})\cdots \tilde{\beta}^{(r_{k})}_{i_{k}}(n_{k})}{i_{1}!\cdots i_{k}!},
$$
$\tilde{\beta}_{k}^{(r)}(n)$ is the reciprocal degenerate Bernoulli of order $r$ given in (\ref{r_RDB}).

Again, by the extended cover-up method we have     
\begin{equation}\label{hj_term}
h_{j}(x) =(1-x)^{r_{j}}\res{\frac{p(x)}{(1-x)^{m+s}}\prod_{\stackrel{i=1}{i\neq j}}^{k}\frac{1}{\Psi_{n_{i}}(x)^{r_{i}}}}{\Psi_{n_{j}}(x)^{r_{j}}},
\end{equation}
where $\textnormal{deg}(h_{j})<n_{j}r_{j}-1$. This polynomial $h_{j}(x)$ can be obtained by Table \ref{eval_table}. 

In order to obtain a Fourier series expansion (not necessarily finite) we write
\begin{equation}\label{hij_term}
\frac{h_{j}(x)}{(1-x^{n_{j}})^{r_{j}}}=\sum^{r_{j}-1}_{i=0}\frac{h_{ij}(x)}{(1-x^{n_{j}})^{r_{j}-i}}=\sum_{i=0}^{r_{j}-1}\frac{1}{(1-x^{n_{j}})^{r_{j}-i-1}}\frac{h_{ij}(x)}{(1-x^{n_{j}})},
\end{equation}
where $\textnormal{deg}(h_{ij})<n_{j}$ and use the finite Fourier series expansion of $h_{ij}(x)/(1-x^{n_{j}})$. By Theorem \ref{taylor_Dm} we can indeed express $h_{j}(x)$ in terms of $h_{ij}(x)$, which is the remainder when $D_{r_{j}}^{i}h_{j}(x)$ is divided by $1-x^{n_{j}}$. Therefore, $\textnormal{deg}(h_{ij})<n_{j}$. In particular, $h_{0j}(x)$ is the remainder when $h_{j}(x)$ is divided by $1-x^{n_{j}}$. This establishes Theorem \ref{main_qpf}.

\subsection{Computing the Denumerants}
We express the denumerants in a concise manner using the piecewise linear function $\lfloor \frac{t}{n}\rfloor$, the greatest integer less than or equal to $\frac{t}{n}$. Suppose $h_{ij}(x)=\sum^{n_{j}-1}_{\alpha=0}c_{\alpha}^{(ij)}x^{\alpha}$. Then, by (\ref{gen_qpf}) and (\ref{hij_term}), the denumerant is given by
$$
d(t;\textbf{A})=\sum^{m+s-1}_{j=0}c_{j}\begin{pmatrix}t+m+s-j-1\\ t \end{pmatrix}+\sum^{k}_{j=1}\sum^{r_{j}-1}_{i=0}\begin{pmatrix}\lfloor \frac{t}{n_{j}}\rfloor +r_{j}-i-1\\ r_{j}-i-1\end{pmatrix}c_{t\%n_{j}}^{(ij)}.
$$
Assuming the multiplications of the polynomials is done using the FFT algorithm, the computational cost of setting up the equations for computing the denumerant requires $O(n\log n)$ steps, where $n=\max\{n_{1},\cdots,n_{k}\}$, which is polynomially bound on the numeric values of the input.

Setting $(r_{1},\cdots,r_{k})=(1,\cdots,1)$ in the above formula we obtain the corresponding denumerant given in (\ref{denumerant_fds}). For the simplest case $p(x)=1$ and $m=0$, we have by Theorem \ref{main_qpf}  and Lemma \ref{pow_k}
$$
c_{j}=\frac{1}{n_{1}\cdots n_{k}}\sum_{i_{1}+\cdots+i_{k}=j} f^{(n_1)}_{i_{1}}(1)\cdots f^{(n_k)}_{i_{k}}(1), \quad \textnormal{ for }0\leq j\leq k-1,
$$
and denoting $h_{j}(x)=\sum^{n_{j}-1}_{i=0}c_{i}^{(j)}x^{i}$ the denumerant is given by 
\begin{equation}\label{denum}
d(t; n_{1},\cdots,n_{k})=\sum^{k-1}_{j=0}c_{j}
\begin{pmatrix}
t+k-j-1\\
t\\
\end{pmatrix}
+\sum^{k}_{j=1}\frac{1}{n_{j}}c^{(j)}_{t\%n_{j}}.
\end{equation}
By this formula one can notice that for the pairwise relatively prime case (and without repetitions) the denumerant is indeed a quasi-polynomial with the periodicity only in the constant term (as observed in \cite{Beck1}).

\begin{exmp}
We compute the denumerant of the set of coprimes $\{9,17,31\}$. So, we consider the $q$-partial fraction of 
$$
F(x)=\frac{1}{(1-x^{9})(1-x^{17})(1-x^{31})}.
$$
Here $k=3$ and $n_{1}=9, n_{2}=17$ and $n_{3}=31$. 
One can easily determine $f_{0}^{(m)}(1)=1$ and $f_{1}^{(m)}(1)=(m-1)/2$.

We first compute $f_{i}^{(n_j)}(1)$ and obtain
$$
f^{(9)}_{0}(1)f^{(17)}_{0}(1)f^{(31)}_{0}(1)=1,\textnormal{ and }  f^{(9)}_{1}(1)+f^{(17)}_{1}(1)+f^{(31)}_{1}(1) = 27.
$$ 
Furthermore, $a_{i}^{(j)}$ for each $j=1,2,3$ satisfying 
$
a_{j}^{(j)}=1 \textnormal{ and }a_{i}^{(j)}n_{i}-b_{i}^{(j)}n_{j}=1,\  a^{(j)}_{i}, b^{(j)}_{i}>0 \textnormal{ for }i\neq j
$  
can be easily computed to obtain 
$$a_{1}^{(1)}= 1,\quad a_{2}^{(1)}=8,\quad a_{3}^{(1)}=7,\quad  a_{1}^{(2)}= 2,\quad  a_{2}^{(2)}=1, \quad  $$
$$a_{3}^{(2)}=11,\quad a_{1}^{(3)}=7,\quad a_{2}^{(3)}=11, \textnormal{ and } a_{3}^{(3)}=1.$$
Substituting for $f_{i}^{(n_j)}(1)$ and $a^{(j)}_{i}$ with simplifications yields:
\begin{eqnarray}
\nonumber g_{0}(x)&=&\left(1+27(1-x)\right),\\
\nonumber g_{1}(x) &=&\left(-2-6x-3x^2-2x^3-3x^4-6x^5-2x^6\right),\\
\nonumber g_{2}(x)&=& \left(13+4x+7x^2+5x^3-2x^4+3x^5+3x^6-2x^7+5x^8+7x^9+4x^{10}+13x^{11}-x^{13}+10x^{14}-x^{15}\right),\\
\nonumber g_{3}(x)&=&\left(14+13x-3x^{2}-3x^{3}+13x^{4}+14x^{5}+2x^{7}-11x^{8}+23x^{9}\right.\\
\nonumber & & \quad\quad +11x^{10}-16x^{11}+4x^{12}+9x^{13}-x^{14}+5x^{15}-4x^{16}+3x^{17}+26x^{18}+3x^{19}\\
\nonumber & & \quad\quad \left. -4x^{20}+5x^{21}-x^{22}+9x^{23}+4x^{24}-16x^{25}+11x^{26}+23x^{27}-11x^{28}+2x^{29}\right).
\end{eqnarray}
Thus the $q$-partial fraction is 
$$
F(x)=\frac{1}{9\cdot17\cdot 31}\frac{g_{0}(x)}{(1-x)^{3}}+ \frac{1}{9}\frac{g_{1}(x)}{(1-x^{9})}+\frac{1}{17}\frac{g_{2}(x)}{(1-x^{17})}+\frac{1}{31}\frac{g_{3}(x)}{(1-x^{31})}.
$$
Further, the denumerant is given by 
$$
d(t;9,17,31)=\frac{1}{4743}\left(\begin{pmatrix}
t+2\\
t\\
\end{pmatrix}+27\begin{pmatrix}
t+1\\
t\\
\end{pmatrix}\right)+\frac{1}{9}c^{(1)}_{t\%9}+\frac{1}{17}c^{(2)}_{t\%17}+\frac{1}{31}c^{(3)}_{t\%31}
$$
where $c^{(j)}_{t\%n_{j}}$ is the $(t\%n_{j})^{th}$ coefficient of $g_{j}(x)$.
\end{exmp}

Although in the above example we have $k=3$, for practical purpose we have considered only two factors $(1-x)^{-2}$ which considerably simplified the calculations. As a consequence of this simplification, one can notice that the coefficients of $g_{j}(x)$, for $1\leq j\leq 3$, do not add up to zero. Furthermore, for the ease of computations, we used $f^{(m)}_{k}(1)$ given in (\ref{fkm_rdb}) instead of $\tilde{\beta}_{k}(m)$.

\section{Generalized Fourier-Dedekind Sum}\label{sec_recip}
Consider the generating function (\ref{transformed_new}) with $r_{1}=\cdots=r_{k}=r$  
\begin{equation}\label{Fx1}
F(x)=\frac{p(x)}{(1-x)^{m}(1-x^{n_1})^{r}\cdots (1-x^{n_k})^{r}}=\sum^{\infty}_{n=0}a(n)x^{n}, 
\end{equation}
where $\textnormal{deg }p(x)=d<m+rs$ and $s=n_{1}+\cdots+n_{k}$. Even if the multiplicities are varying, one can always multiply and divide appropriately by a $p(x)$ to ensure $r_1=\cdots=r_k=r$. For example, in (\ref{exmp_5_case}) we can have $$
\prod^{5}_{j=1}\frac{1}{1-x^{j}}=\frac{(1+x^{2})(1-x^3)(1-x^5)}{(1-x)(1-x^{3})^{2}(1-x^{4})^{2}(1-x^{5})^{2}}.
$$
Qualitatively speaking, this will not lead to anything new.

Taking motivation from the Substitution Rule 2 (Lemma \ref{eval_eval}) we state the following generalization of the Fourier-Dedekind sum. 

\begin{defn}\label{FDS}
For a positive integer $b$ that is coprime to $n_{j}$, for $1\leq j\leq k$, the $k$-dimensional generalized Fourier-Dedekind sum $k>0$ is 
\begin{equation}\label{gen_FDS}
S^{(m, r, p(x))}_{t}(n_{1},\cdots,n_{k}; b)=\frac{1}{b^{r}}\sum^{b-1}_{j=1}\frac{p(\xi_{b}^{j})\xi_{b}^{jt}}{(1-\xi_{b}^{j})^{m}(1-\xi_{b}^{jn_{1}})^{r}\cdots(1-\xi_{b}^{jn_{k}})^{r}}
\end{equation}
and for $k=0$ we define 
\begin{equation}\label{trivial_gen_FDS}
S^{(m, r, p(x))}_{t}(b)=\frac{1}{b^{r}}\sum^{b-1}_{j=1}\frac{p(\xi_{b}^{j})\xi_{b}^{jt}}{(1-\xi_{b}^{j})^{m}},
\end{equation}
where $\xi_{b}=e^{2\pi i/b}$. 
\end{defn} 
This is a generalization of the Fourier-Dedekind sum given in \cite{Beck1,Beck,Tuskerman,Zagier,Gessel}. Indeed, if we set $p(x)=1, r=1$ and $m=0$ we recover the Fourier-Dedekind sum in \cite{Beck,Tuskerman}
\begin{equation}\label{classical_FDS}
s_{t}(n_{1},\cdots,n_{k}; b)=\frac{1}{b}\sum^{b-1}_{j=1}\frac{\xi_{b}^{jt}}{(1-\xi_{b}^{jn_{1}})\cdots(1-\xi_{b}^{jn_{k}})},
\end{equation}
that is, $S^{(0,1,1)}_{t}=s_{t}$. Reciprocity theorem for $p(x)=1, m=1$ was considered in \cite{Beck1}. 

\begin{rema}
By definition we have $S^{(m,1,x^{\alpha})}_{t}=S^{(m,1,1)}_{t+\alpha}$.  Thus, suppose $p(x)=\sum^{s}_{\alpha=0}a_{\alpha}x^{\alpha}$ then we can write the generalized Fourier-Dedekind sum as a linear combination $S^{(m,r,p(x))}_{t}=\sum^{s}_{\alpha=0}a_{\alpha}S^{(m,r,1)}_{t+\alpha}$. Nevertheless, as it happens in our case, if the coefficients of $p(x)$ are not known apriori (for instance, if $p(x)=\textnormal{eval}(f(x)/g(x);\Psi_{m}(x))$) then one can't make such simplifications. 
\end{rema}Following are some easy to prove properties.  
\begin{enumerate}
\item $S^{(m,r,p(x))}_{t}(n_{1},\cdots,n_{k};b )$ is rational and is symmetric in $n_{1},\cdots, n_{k}$. 
\item $S^{(m,r,p(x))}_{t}(n_{1},\cdots,n_{k};b )$ only depends on $n_{i}\mod b$. 
\item $S^{(m,r, p(x))}_{t}(\lambda n_{1},\cdots,\lambda n_{k};b )=S^{(m,r,p(x))}_{t}(n_{1},\cdots,n_{k};b )$ if $\textnormal{gcd}(\lambda,b)=1$. 
\item $S^{(m,1,p(x))}_{t}(n_{1},\cdots,n_{k};b )=s_{t}(n_{1},\cdots,n_{k};b )*_{t}S^{(m,1,p(x))}_{t}(b)$. (Follows from \cite[Theorem 7.10]{Beck}.)\label{prop_4}
\end{enumerate}
 
By reciprocity law we mean identities for certain sums of generalized Fourier-Dedekind sums. In the literature, there are essentially two types of reciprocity laws in relation to the Dedekind sums: Zagier ($t=0$) and Rademacher ($1\leq t <\lambda$, for some $\lambda$) \cite{Beck}. 

Property (\ref{prop_4}) suggests that the reciprocity law satisfied by $s_{t}$ also holds for $S^{(m,1,p(x))}_{t}$ (i.e. for $r=1$ case). So, we first set forth to prove the reciprocity result for the case $r=1$ and $m\geq 0$ in (\ref{Fx1}). Suppose the formal power series for  $F(x)$ is given $\sum^{\infty}_{t=0}a(t)x^{t}$. Then we have the coefficient $a(t)$ from (\ref{denumerant_fds})
$$
a(t) = \sum^{m+k-1}_{j=0}c_{j}
\begin{pmatrix}
t+m+k-j-1\\
t\\
\end{pmatrix}+\sum_{j=1}^{k}S^{(m, 1, p(x))}_{-t}(n_{1},\cdots,\widehat{n_{j}},\cdots,n_{k}; n_{j}),
$$
where $\widehat{ }$ refers to dropping the term. Thus we can write   
\begin{equation}\label{at_term}
a(t) =  \textnormal{poly}(t) +\quad \sum^{k}_{j=1} S^{(m,1,p(x))}_{-t}(n_1,\cdots,\widehat{n_{j}},\cdots,n_{k};\ n_{j}),
\end{equation}
where 
\begin{equation}\label{poly1}
\textnormal{poly}(t)= \sum_{j=0}^{m+k-1}c_{j}\begin{pmatrix}t+m+k-j-1 \\ t\end{pmatrix}.
\end{equation}
Since the sequence $(a(t))$ is associated with a proper rational generating function (\ref{Fx1}) one can easily demonstrate that it satisfies a linear recurrence relation dependent on the first few terms $a(0),\cdots, a(n)$. In order to obtain a reciprocity relation we consider the formal power series 
$$
F^{o}(x)=\sum^{\infty}_{t=1}a(-t)x^{t},
$$
where the coefficients $a(-t)$ are derived by ``running backward'' the recurrence relation satisfied by $(a(t))$. By the result \cite[Theorem 4.1.6]{Beck3}, we have the relation $F^{o}(x)=-F(1/x)$. Therefore, by (\ref{Fx1}), we obtain   
$$
F^{o}(x)=-F(1/x)=\frac{(-1)^{m+k-1}x^{s+m-d}\tilde{p}(x)}{(1-x)^{m}(1-x^{n_1})\cdots (1-x^{n_k})},
$$
where $\tilde{p}(x)=x^{d}p(1/x)$. Due to the factor $x^{s+m-d}$ in the numerator of $F^{o}(x)$ we have   
\begin{equation}\label{at_minus}
a(-t)=0 \quad \textnormal{   if   }\quad 1 \leq t < s+m-d. 
\end{equation}
\begin{figure}[t!]
\centering 
\includegraphics[scale=0.4]{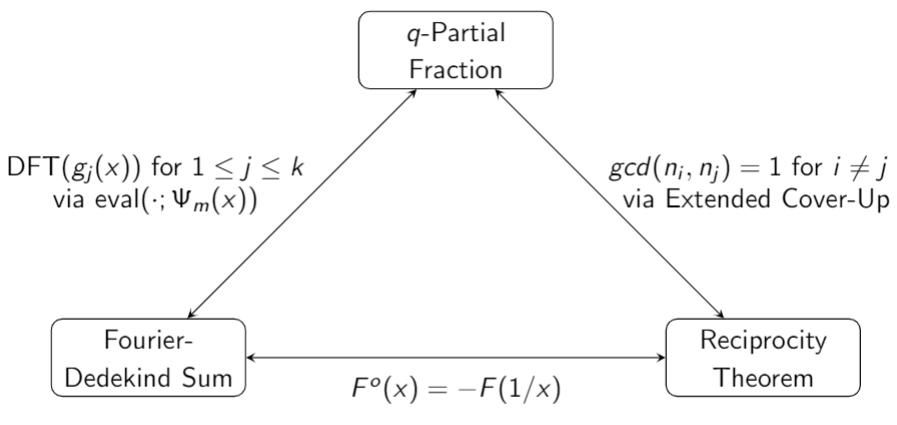}
\\ \textit{Scheme for the Reciprocity Theorem} \\ $F(x)=\frac{p(x)}{(1-x)^{m}(1-x^{n_{1}})\cdots(1-x^{n_{k}})}=\frac{g_{0}(x)}{(1-x)^{m+k}}+\sum^{k}_{j=1}\frac{g_{j}(x)}{(1-x^{n_{j}})}$.
\end{figure}  

\begin{thm}[Reciprocity Theorem, $r=1$ and $m\geq 1$]\label{k_dim_recip_thm}
For all pairwise relatively prime positive integers $n_{1},\cdots, n_{k}$, and letting $r=1$ in the generating function (\ref{Fx1}), the equation
$$
\sum^{k}_{j=1} S^{(m,1,p(x))}_{t}(n_1,\cdots,\widehat{n_{j}},\cdots,n_{k};\ n_{j})= 
      \left\{\begin{array}{@{}l@{\thinspace}l}
         p(0)-\textnormal{poly}(0) & \textnormal{   if   } t=0 \\
       -\textnormal{poly}(-t)  & \textnormal{   if   }1\leq t<s+m-d
     \end{array}\right.,
$$
holds, where $\textnormal{poly}(t)$ is given in (\ref{poly1}) and $s=n_{1}+\cdots + n_{k}$.
\end{thm}

\begin{proof}[Proof.]
The case $t=0$ can be obtained directly by substituting $a(0)=p(0)$ into the equation (\ref{at_term}). For the other case replace $t$ by $-t$ into (\ref{at_term}) and use (\ref{at_minus}). Further we use the identity (see \cite{Beck})
\begin{equation}\label{comb_n}
(-1)^{l}\begin{pmatrix}
-t+\lambda\\
l\\
\end{pmatrix}=\begin{pmatrix}
t+l-1-\lambda\\
l\\
\end{pmatrix}, 
\end{equation}
for $t,\lambda\in \mathbb{Z} \textnormal{ and } l\in \mathbb{Z}^{+},$ to obtain the required result.   
\end{proof}
  
\begin{exmp}
Suppose $k$ is odd. Then a generalization of the Zagier's higher dimensional Dedekind sum is
$$
Z_{t}^{(m)}(n_{1},\cdots,n_{k}; b)=\frac{2}{b}\sum^{b-1}_{j=1}\frac{(1+\xi_{b}^{jn_{1}})\cdots(1+\xi_{b}^{jn_{k}})\xi_{b}^{-jt}}{(1-\xi_{b}^{j})^{m}(1-\xi_{b}^{jn_{1}})\cdots(1-\xi_{b}^{jn_{k}})},
$$
for  $\xi_{b}=e^{2i\pi/b}$. The motivation for the trigonometric sum with $t=m=0$ was given in \cite{Zagier} by topological considerations and no generating function was mentioned. By a direct observation from the formulas of the generalized Fourier-Dedekind sum a numerator for the generating function is 
$$
\theta(x)=(1+x^{n_{1}})\cdots(1+x^{n_{k}})+(1-x^{n_{1}})\cdots(1-x^{n_{k}}).
$$
 Upon a simplification we have 
$ 
 \theta(x)=2\sum_{S\subset A}x^{|S|},  
$
where $S$ is a subset of  $A=\{n_{1},\cdots,n_{k}\}$ with even cardinality.  
Therefore 
$
Z_{t}^{(m)}=S^{(m,1,\theta(x))}_{t}.
$
\end{exmp}

The $q$-partial fractions decomposition is, in general, not unique. A special gcd property is satisfied by the coefficients obtained through our $q$-partial fractions and the coefficients given in the formula (\ref{denum}). 

\begin{thm} The reciprocity theorems imply the following conditions on the coefficients. 
\begin{enumerate}
\item (Rademacher Reciprocity Theorem $k=2$)
$$
n_{2}c^{(n_{1})}_{-n\%n_{1}}+n_{1}c^{(n_{2})}_{-n\%n_{2}}= \left\{\begin{array}{@{}l@{\thinspace}l}
       n-1 & \text{   if   } 1\leq n<n_{1}n_{2} \\
       n_{1}n_{2}-1  & \text{   if   }n=0. \\
     \end{array}\right.
$$

\item (Sylvester Reciprocity Theorem, $k=2$)
$$
\left(\frac{c^{(n_1)}_{n\%n_{1}}+c^{(n_1)}_{-n\%n_{1}}}{-2}\right)n_{2}+\left(\frac{c^{(n_2)}_{n\%n_{2}}+c^{(n_2)}_{-n\%n_{2}}}{-2}\right)n_{1}=1,
$$
for $1\leq n<n_{1}n_{2}$.

\item (Rademacher Reciprocity Theorem, $k=3$)
$$
c^{(n_{1})}_{-n\%n_{1}}n_{2}n_{3}+c^{(n_{2})}_{-n\%n_{2}}n_{1}n_{3}+c^{(n_{3})}_{-n\%n_{3}}n_{1}n_{2}= \left\{\begin{array}{@{}l@{\thinspace}l}
       \frac{(n-1)(n+1-s)}{2} & \text{   if   } 1\leq n<s' \\
       2s'-s+1  & \text{   if   }n=0. \\
     \end{array}\right.
$$
where $s=n_{1}+n_{2}+n_{3}$ and $s'=n_{1}n_{2}n_{3}$.

\end{enumerate}
\end{thm}
\begin{proof}
The results follow directly from Theorem \ref{k_dim_recip_thm} and the Sylvester reciprocity theorem \cite[Lemma 1.7]{Beck}  
$$
\quad\quad d(t; n_{1},n_{2})+d(n_{1}n_{2}-t;n_{1},n_{2})=1,\quad \textnormal{ for } 1\leq t<n_{1}n_{2}.
$$ 
 
\end{proof}

Another direct consequence of the reciprocity theorem is showing that certain trigonometric sums are indeed polynomials. For instance, by taking the $0$-dimensional reciprocity relation for (\ref{pf_k}) in Lemma \ref{pow_k} we have 
\begin{equation}\label{sigma_k}
1=\frac{1}{m}\sum_{j=0}^{k}f^{(m)}_{j}(1)+\frac{1}{m}\sum^{m-1}_{j=1}\frac{1}{(1-\xi^{j})^{k}}.
\end{equation}
Therefore,  $\sum_{j=1}^{m-1} \frac{1}{(1-\xi^{j})^{k}}$ is a polynomial provided $f^{(m)}_{k}(1)$ is a polynomial. From the equation (\ref{sigma_k}) we have 
$$
f^{(m)}_{k}(1)=\sum_{j=1}^{m-1}\left(\frac{1}{(1-\xi^{j})^{k-1}}-\frac{1}{(1-\xi^{j})^{k}}\right)=(-1)^{k}\frac{\tilde{\beta}_{k}(m)}{k!}.
$$
Gessel \cite{Gessel} gave an explicit formula for $\sum \frac{1}{(1-\xi^{j})^{k}}$ using analytic methods.

\begin{thm}[Gessel \cite{Gessel}]\label{fjm_polynomial}
\begin{equation}\label{fjm}
\frac{(-1)^{k}}{k!}\tilde{\beta}_{k}(m) = \left\{\begin{array}{@{}l@{\thinspace}l}
       \frac{m-1}{2} & \textnormal{   if   } k=1 \\
       \frac{1}{(k-1)!}\sum^{k}_{j=2}\frac{B_{j}}{j}\textnormal{stirl}(k-1,j-1)(m^{j}-1)  & \textnormal{   if   } k\geq 2. \\
     \end{array}\right.
\end{equation}
\end{thm}

\subsection{$\tilde{\beta}_{k}(m)$ is a Polynomial in $m$: An Algebraic Proof }\label{computation}
We provide a purely algebraic proof that the reciprocal polynomial of the degenerate Bernoulli number $\tilde{\beta}_{k}(m)$ is a polynomial in $m$, where 
$$
\frac{(-1)^k}{k!}\tilde{\beta}_{k}(m,x)= \left(-\frac{1}{m}x\Psi'_{m}(x)\right)^{k} \textnormal{ rem } \Psi_{m}(x).
$$
Towards this direction, we denote $g_{k}(x)=\left(-x\Psi'_{m}(x)\right)^{k} \textnormal{ rem } \Psi_{m}(x)$ and show that $g_{k}(1)$ is a polynomial in $m$ with a factor $m^{k}$. For the purpose of proving this claim we use the well known fact that $m(m+1)$ is a factor of the sum of powers $\sigma_{n}(m)=\sum_{i=1}^{m}i^{n}$.  

\begin{lem}\label{mk_factor}
The function $\left(xD\right)^{n}g_{k}(x)$, for $n=0,1,2,\cdots,$ when evaluated at $x=1$ is a polynomial in $m$ with zero as a root of multiplicity at least $k$. 
\end{lem}
\begin{proof}[Proof.]
We prove by induction on $k$. The result holds for $k=1$ case. Indeed, by a direct calculation we have 
$$
g_{1}(x)=x^{m-2}+\cdots+(m-2)x+(m-1),
$$
which implies $g_{1}(1)=m(m-1)/2$ and
$$
\left(xD\right)^{n}g_{1}(x)=(m-2)^{n}\cdot 1 x^{m-2}+(m-3)^{n}\cdot 2 x^{m-3}+\cdots+1^{n}\cdot(m-2)x 
$$
implies
\begin{eqnarray}
\nonumber \left.\left(xD\right)^{n}g_{1}(x)\right\vert_{x=1}&=&(m-2)^{n}\cdot 1+(m-3)^{n}\cdot 2+\cdots+1^{n}\cdot (m-2)\\
\nonumber &=& (m-1)\sigma_{n}(m-1)-\sigma_{n+1}(m-1),
\end{eqnarray}
which is also divisible by $m$.

For the induction step, suppose $g_{k}(x)=\alpha_{m-2}x^{m-2}+\cdots+\alpha_{1}x+\alpha_{0}$, where the coefficients $\alpha_{0},\cdots,\alpha_{m-2}$ are dependent on $m$ and $k$, and assume $\left.\left(xD\right)^{n}g_{k}(x)\right\vert_{x=1}=\sum_{t=0}^{m-2}t^{n}\alpha_{t}$ have $m^{k}$ as a factor for $n=0,1,\cdots$. \\
Consider the $(k+1)$ case
\begin{eqnarray}
\nonumber g_{k+1}(x)&=& \left(g_{k}(x)(-x\Psi'_{m}(x))\right) \textnormal{ rem }\Psi_{m}(x)\\
\nonumber &=& \left(\sum_{i,j=0,0}^{m-2,m-2}(m-1-j)\alpha_{i}x^{i+j}\right) \textnormal{ rem } \Psi_{m}(x)\\
\nonumber &=& \left(\sum_{\stackrel{i,j=0,0}{(i+j)\%m\neq(m-1)}}^{m-2,m-2}\alpha_{i}(m-1-j)x^{(i+j)\% m}\right) +   \left(\sum_{\stackrel{i,j=0,0}{(i+j)\%m=(m-1)}}^{m-2,m-2}\alpha_{i}(m-1-j)(x^{m-1}-\Psi_{m}(x))\right),
\end{eqnarray}
here we substitute for $x^{m-1}$ by $(x^{m-1}-\Psi_{m}(x))$ using the Substitution Rule 1 (Lemma \ref{rem_psi}). Simplifying we obtain
$$
g_{k+1}(x)=\sum_{i,j=0,0}^{m-2,m-2}\alpha_{i}(m-1-j)x^{(i+j)\% m}-(1+x+\cdots+x^{m-1})\sum_{t=1}^{m-2}t\alpha_{t}.
$$
and substituting $x=1$ 
\begin{eqnarray}
\nonumber g_{k+1}(x) \nonumber &=&\sum_{i,j=0,0}^{m-2,m-2}\alpha_{i}(m-1-j)-m\sum_{t=1}^{m-2}t\alpha_{t}\\
\nonumber &=&m\left(\sum_{i=0}^{m-2}\frac{m-1}{2}\alpha_{i}-\sum_{t=1}^{m-2}t\alpha_{t}\right)= m^{k+1}\times (\textnormal{polynomial in } m),
\end{eqnarray}
where the last statement holds by the induction hypothesis. 

Now we show that $\left.\left(xD\right)^{n}g_{k+1}(x)\right\vert_{x=1}$ also has $m^{k+1}$ as a factor. Observe that  
$
\left.\left(xD\right)^{n}g_{k}(x)\right\vert_{x=1}= \sum_{t=0}^{m-2}t^{n}\alpha_{t}
$
and 
\begin{equation}\label{gk_1}
\left.\left(xD\right)^{n}g_{k+1}(x)\right\vert_{x=1}=\sum_{t=0}^{m-2}\left(\sum_{i=1}^{t-2}i^{n}(t-1-i)+\sum_{i=k}^{m-1}i^{n}(m-1+t-i)\right)\alpha_{t}+ \sigma_{n}(m-1)\sum_{t=0}^{m-2}t^{n}\alpha_{t}.    
\end{equation}

Fortunately, we don't need to determine a formula for the summations; instead, we just need to show that $m$ is a factor.  The inner sum in the right hand side of the first term of the equation (\ref{gk_1}) can be simplified to
$$
(m-1)t^{n}+\cdots+t(m-1)^{n}=(m-1+t)\left\{\sigma_{n}(m-1)-\sigma_{n+1}(t-1)\right\}-\left\{\sigma_{n+1}(m-1)-\sigma_{n+1}(t-1)\right\}.
$$
Furthermore, setting $m=0$ we get 
$-(t-1)\sigma_{n+1}(t-1)+\sigma_{n+1}(t-1).$ Also, the second term is 
\begin{eqnarray}
\nonumber (t-2)\cdot 1^{n}+\cdots+1\cdot (t-2)^{n}&=&(t-1)\sigma_{n}(t-1)-\sigma_{n+1}(t-1)
\end{eqnarray}
Thus, the factor $(m-1)t^{n}+\cdots+t(m-1)^{n}+(t-2)\cdot 1^{n}+\cdots+1\cdot (t-2)^{n}$ vanishes when $m=0$. Therefore, the equation (\ref{gk_1}) is of the form $\left.\left(xD\right)^{n}g_{k}(x)\right\vert_{x=1}=m\left(\sum_{t=0}^{m-2}q(m,t)\alpha_{t}\right),$ where $q(m,t)$ is a polynomial in $m$ and $t$. Hence, by the induction hypothesis $\left.\left(xD\right)^{n}g_{k+1}(x)\right\vert_{x=1}$ has $m^{k+1}$ factor.  
\end{proof}

\section{Sylvester Waves}\label{sylvester_waves}

\subsection{Polynomial Part of $d(t;\textbf{A})$: $W_{1}(t;\textbf{A})$}
Given the arbitrary $k$-tuple $\textbf{A}=(a_{1},\cdots, a_{k})$ of positive integers, following our methodology we can state an explicit formula for the polynomial part $W_{1}(t;\textbf{A})$ of $d(t;\textbf{A})$. By the $q$-partial fraction (\ref{gen_qpf}), the polynomial part $W_{1}(t;\textbf{A})$ can be obtained by first simplifying the following expression using (\ref{young_eqn})  
\begin{eqnarray}
\nonumber g(x)=\textnormal{eval}\left(\frac{1}{\Psi_{a_{1}}(x)\cdots \Psi_{a_{k}}(x)};(1-x)^{k}\right) &=& \left(\prod_{j=1}^{k}\textnormal{eval}\left(\frac{1}{\Psi_{a_{j}}(x)};(1-x)^{k}\right)\right) \textnormal{ rem } (1-x)^{k}\\
\nonumber &=& \left(\prod_{j=1}^{k}\frac{1}{a_{j}}\sum_{i=0}^{k-1}(-1)^{i}\frac{\tilde{\beta}_{i}(a_{j})}{i!}(1-x)^{i}\right) \textnormal{ rem }(1-x)^{k}.
\end{eqnarray}
Therefore, the polynomial part corresponds to the coefficient of $x^{t}$ in  
$
g(x)/(1-x)^{k} 
$ is given by
\begin{equation}\label{sec_W1}
W_{1}(t;\textbf{A})=\frac{1}{a_{1}\cdots a_{k}}\sum^{k-1}_{j=0}\sum_{j_{1}+\cdots+j_{k}=j}(-1)^{j}\begin{pmatrix}t+k-j-1 \\ t\end{pmatrix}\frac{\tilde{\beta}_{j_{1}}(a_{1})\cdots \tilde{\beta}_{j_{k}}(a_{k})}{j_{1}!\cdots j_{k}!}.
\end{equation}
Our formula can be computed just by polynomial arithmetic using  (\ref{fkm_rdb}) or by using Bernoulli and Stirling numbers (see  Theorem \ref{fjm_polynomial}).
\begin{rema}
Since $\tilde{\beta}_{k}(n)$, given in (\ref{young_eqn}), is defined by an exponential generating function one can express the formula (\ref{sec_W1}) in a simpler form through the notation of umbral calculus \cite{gessel_umbral}
$$
W_{1}(t;\textbf{A})=\frac{1}{a_{1}\cdots a_{k}}\sum^{k-1}_{j=0}\begin{pmatrix}t+k-j-1 \\ t\end{pmatrix}\frac{(-1)^{j}}{j!}\left\{\tilde{\beta}(a_{1})+\cdots+\tilde{\beta}(a_{k})\right\}^{j},
$$
where after the multinomial expansion $\tilde{\beta}(a_{i})^{j}$ is replaced with $\tilde{\beta}_{j}(a_{i})$.
\end{rema}

\subsection{Periodic Parts of $d(t;\textbf{A})$: $W_{n_{j}}(t;\textbf{A})$}

We now consider the decomposition of the general case 
\begin{equation}\label{denum_decompose}
d(t;\textbf{A})=W_{1}(t;\textbf{A})+\sum^{k}_{j=1}W_{n_{j}}(t;\textbf{A}),
\end{equation}
where $\{n_{1},\cdots,n_{k}\}$ corresponds to (\ref{transformed_new}). By Theorem \ref{main_qpf}, we have a $q$-partial fraction
\begin{equation}\label{qpf_case2_eqn}
\frac{p(x)}{(1-x)^{m}}\prod_{j=1}^{k}\frac{1}{(1-x^{n_{j}})^{r_{j}}}=\sum^{m+s-1}_{j=0}\frac{c_{j}}{(1-x)^{m+s-j}}+\sum^{k}_{j=1}\sum_{i=0}^{r_{j}-1}\frac{h_{ij}(x)}{(1-x^{n_{j}})^{r_{j}-i}},
\end{equation}
where $s=r_{1}+\cdots+r_{k}$ and $\textnormal{deg}(h_{ij})<n_{j}$ for all $j=1,\cdots,k$. In order to determine the Fourier series we first evaluate $h_{ij}(x)$, given in (\ref{hij_term}), at $\xi_{j}=e^{2\pi i/n_{j}}$. 

\begin{thm}\label{eval_hj}
Evaluation of (\ref{hj_term}) and (\ref{hij_term}) at $\xi_{j}=e^{2\pi i /n_{j}}$ gives us
$$
h_{ij}(\xi_{j})=\frac{(-1)^{i}}{i!}D_{n_{j}}^{i}h_{j}(\xi_{j}),
$$
where $D_{n}(x^{\alpha})=\lfloor\frac{\alpha}{n}\rfloor x^{\alpha-n}$. In particular, we have 
\begin{equation}\label{h0j_term}
h_{0j}(\xi_{j})=\frac{p(\xi_{j})}{(1-\xi_{j})^{m}(1-\xi_{j}^{n_{1}})^{r_{1}}\cdots \widehat{(1-\xi_{j}^{n_{j}})^{r_{j}}}\cdots (1-\xi_{j}^{n_{k}})^{r_{k}}}.
\end{equation}
\end{thm}
\begin{proof}[Proof.]
The proof easily follows from Equation \ref{gen_hij} and Substitution Rule 2, Lemma \ref{eval_eval}.  
\end{proof}

To obtain a direct formula for the Sylvester decomposition in the $q$-partial fraction in Theorem \ref{main_qpf} it is sufficient to consider a rational function of the form 
\begin{equation}\label{hx_nr}
\frac{h(x)}{(1-x^{n})^{r}}=\sum^{r-1}_{i=0}\frac{(-1)^{i}}{i!}\textnormal{eval}(D^{i}_{n}h(x); 1-x^{n})\frac{1}{(1-x)^{r-i}}, 
\end{equation}
with $\textnormal{deg}(h)<nr$, for positive integers $n$ and $r$. The Fourier series of each term of the above expression can be written as a finite product of finite Fourier series. That is,  
\begin{equation}\label{decomposition_term}
\frac{(-1)^{i}}{i!(1-x)^{r-i}}\textnormal{eval}(D^{i}_{n}h(x); 1-x^{n}) = \left(\sum^{\infty}_{t=0}\frac{1}{n}\sum^{n-1}_{j=0}\xi_{n}^{jt}x^{t}\right)^{r-i-1}\left(\sum^{\infty}_{t=0}\frac{1}{n}\sum^{n-1}_{j=0}\frac{(-1)^{i}}{i!}D_{n}^{i}h(\xi_{n}^{j})\xi_{n}^{jt}x^{t}\right),
\end{equation}
where $\xi_{n}=e^{2\pi i/n}$. Hence, the coefficients are functions of $\xi_{n}$, establishing the Sylvester decomposition. 

Suppose $h(x)$ in  (\ref{hx_nr}) is denoted by $\sum^{n-1}_{j=0}a(j)x^{j}$. Then, the top-order term of $t^{th}$ coefficient in the formal power series of (\ref{hx_nr}) is given by      
\begin{eqnarray}
\nonumber \begin{pmatrix}\lfloor\frac{t}{n}\rfloor + r - 1 \\ r-1 \end{pmatrix}a(t) &=& \frac{a(t)}{(r-1)!}\left(\left\lfloor\frac{t}{n}\right\rfloor+r-1\right)\cdots\left(\left\lfloor\frac{t}{n}\right\rfloor+1\right)\\
\nonumber &=& \frac{a(t)}{(r-1)!}\left\lfloor\frac{t}{n}\right\rfloor^{r-1}+O\left(\left\lfloor\frac{t}{n}\right\rfloor^{r-2}\right).
\end{eqnarray}
Here $\left\lfloor\frac{t}{n}\right\rfloor$ is the largest integer less than or equal to $\frac{t}{n}$. The function $\left\lfloor\frac{t}{n}\right\rfloor$ is given by the trigonometric sum (see \cite[pg. no. 10]{Beck})
$$
\left\lfloor\frac{t}{n}\right\rfloor=\frac{t}{n}+\frac{1}{2n}+\frac{1}{2}+\frac{1}{n}\sum^{n-1}_{j=1}\frac{\xi_{n}^{-jt}}{1-\xi_{n}^{j}},
$$
for $\xi_{n}=e^{2\pi i/n}$. Thus, the coefficient of $t^{th}$ term of (\ref{hx_nr}) is 
\begin{equation}\label{t_term}
\frac{a(t)}{(r-1)!}\left(\frac{t}{n}\right)^{r-1}+O(t^{r-2}).
\end{equation} 
We are now in a position to prove Theorem \ref{structure_theorem}.

\begin{proof}[Proof of Theorem \ref{structure_theorem}.]
The decomposition (\ref{denum_decompose}) follows from (\ref{main_qpf_formula}) and (\ref{decomposition_term}). The expression (\ref{wave_j}) for $W_{n_{j}}(t;\textbf{A})$ can be obtained from the top-order term (\ref{t_term}) from the Fourier series expansion (\ref{zero_term}) and (\ref{h0j_term}).  
\end{proof}

\begin{exmp}\label{n1n1n2}
Consider $\textbf{C}=(n_{1},n_{1},n_{2})$ with $\textnormal{gcd}(n_{1},n_{2})=1$. The polynomial part $W_{1}(t;\textbf{C})$ is the coefficient of $x^{t}$ in the formal power series of 
\begin{eqnarray}
\nonumber \frac{1}{(1-x)^{3}}\textnormal{eval}\left(\frac{1}{\Psi_{n_{1}}(x)^{2}\Psi_{n_{2}}(x)}; (1-x)^{3}\right)&=&  \sum^{2}_{j_{1}+j_{2}+j_{3}=0}\frac{c_{j_{1}+j_{2}+j_{3}}}{(1-x)^{3-(j_{1}+j_{2}+j_{3})}},
\end{eqnarray}
where 
$$
c_{j}=\frac{1}{n_{1}^{2}n_{2}}\sum_{j_{1}+j_{2}+j_{3}=j}f_{j_{1}}^{(n_{1})}(1)f_{j_{2}}^{(n_{1})}(1)f_{j_{3}}^{(n_{2})}(1).
$$
From (\ref{fkm_rdb}) and (\ref{fjm}) we can easily deduce 
$$
f_{0}^{(m)}(1)=1,\quad  f_{1}^{(m)}(1)=\frac{m-1}{2},\quad
\textnormal{ and }\quad 
f_{2}^{(m)}(1)=\frac{m^2-1}{12}.
$$
The wave $W_{n_{1}}(t;\textbf{C})$ is the coefficient of $x^{t}$ in 
$$
\frac{g_{1}(x)}{(1-x^{n_{1}})^{2}}=\frac{(1-x)^{2}}{(1-x^{n_{1}})^{2}}\textnormal{eval}\left(\frac{1}{(1-x)^{2}(1-x^{n_{2}})};\Psi_{n_{1}}(x)^{2}\right).
$$
The coefficient of $W_{n_{2}}(t;\textbf{C})$ is the coefficient of $x^{t}$ in 
$$
\frac{g_{2}(x)}{1-x^{n_{2}}}=\frac{(1-x)}{(1-x^{n_{1}})}\textnormal{eval}\left(\frac{1}{(1-x)(1-x^{n_{2}})^{2}};\Psi_{n_{2}}(x)\right).
$$
Using the expression for $\textnormal{eval}$ given in Table \ref{eval_table} and the properties of the $\textnormal{eval}$ given in Lemma \ref{lem_res} to simplify the above expressions we have the Sylvester waves given by\begin{eqnarray}
\nonumber W_{1}(t)&=&\frac{1}{2n_{1}^{2}n_{2}}t^{2}+\frac{2n_{1}+n_{2}}{2n_{1}^{2}n_{2}}t+\frac{5n_{1}^{2}+n_{2}^{2}+6n_{1}n_{2}}{12
n_{1}^{2}n_{2}},\\
\nonumber W_{n_{1}}(t)&=&\frac{1}{n_{1}^{2}}\left(\sum^{n_{1}-1}_{j=1}\frac{\xi_{n_{1}}^{-jt}}{1-\xi_{n_{1}}^{jn_{2}}}\right)t+\left(\frac{1}{2n_{1}^{2}}+\frac{1}{n_{1}^{2}}\sum^{n_{1}-1}_{j=1}\frac{\xi_{n_{1}}^{-jt}}{1-\xi_{n_{1}}^{jn_{2}}}\right)\left(\sum^{n_{1}-1}_{j=1}\frac{\xi_{n_{1}}^{-jt}}{1-\xi_{n_{1}}^{jn_{2}}}\right)+O(1)\quad \textnormal{ and }\\
\nonumber W_{n_{2}}(t)&=&\frac{1}{n_{2}}\sum^{n_{2}-1}_{j=1}\frac{\xi_{n_{2}}^{-jt}}{(1-\xi_{n_{2}}^{jn_{1}})^{2}}.
\end{eqnarray}
Furthermore, by $\xi_{n_{1}}^{j(t-n_{1})}=\xi_{n_{1}}^{jt}$, the denumerant $d(t;\textbf{C})=d(t;n_{1},n_{1},n_{2})=W_{1}(t)+W_{n_{1}}(t)+W_{n_{2}}(t)$ verifies  the Euler's recurrence formula
\begin{eqnarray}
\nonumber d(t; n_{1},n_{1},n_{2})&-&d(t-n_{1}; n_{1},n_{1},n_{2})\\
\nonumber &=& \frac{2t+n_{1}+n_{2}}{2n_{1}n_{2}}+\frac{1}{n_{1}}\sum^{n_{1}-1}_{j=1}\frac{\xi_{n_{1}}^{-jt}}{1-\xi_{n_{1}}^{jn_{2}}}+\frac{1}{n_{2}}\sum^{n_{2}-1}_{j=1}\frac{\xi_{n_{2}}^{-jt}}{1-\xi_{n_{2}}^{jn_{1}}}\\
\nonumber &=&d(t;n_{1},n_{2}).
\end{eqnarray} 
\end{exmp}

\subsection{The case $r_{1}=\cdots=r_{k}=r$.}
When $r_{1}=\cdots=r_{k}=r$, by Theorem \ref{structure_theorem} and Definition \ref{FDS}, we can see that the top-order term of each wave $W_{n_{j}}(t;\textbf{A})$ takes the form of the generalized Fourier-Dedekind sum.  That is, 
\begin{equation}\label{Wnj_S}
W_{n_{j}}(t;\textbf{A})=S_{-t}^{(m,r,p(x))}(\textbf{A}_{j};n_{j})\frac{t^{r-1}}{(r-1)!}+O(t^{r-2}),
\end{equation}
where $\textbf{A}_{j}=(n_{1},\cdots,\widehat{n_{j}},\cdots, n_{k})$. Thus, in this case, we can see a beautiful interplay between the denumerant and the generalized Fourier-Dedekind sum. Adding all the waves (\ref{Wnj_S}), for $j=1,\cdots,k$, and from the decomposition (\ref{denum_decompose}) we have 
\begin{eqnarray}
\nonumber d(t;\textbf{A}) &=& W_{1}(t;\textbf{A})+\sum^{k}_{j=1}W_{n_{j}}(t;\textbf{A})\\
&=&W_{1}(t;\textbf{A})+\left(\sum^{k}_{j=1}S_{-t}^{(m,r,p(x))}(\textbf{A}_{j};n_{j})\right)\frac{t^{r-1}}{(r-1)!}+ O(t^{r-2}).\label{denum_recip}
\end{eqnarray}

The above equation motivates us to seek a reciprocity theorem satisfied by the generalized Fourier-Dedekind sum $S_{t}^{(m,r,p(x))}$. In order to achieve this we need to invoke a formal reciprocity relation on an appropriate rational generating function. Thus our first task is to identify a rational generating function $G(x)$ with a formal power series $\sum_{t=0}^{\infty}b(t)x^{t}$, which satisfies a linear recurrence relation for some numbers $c_{0},\cdots, c_{n}$ with $c_{0},c_{n}\neq 0$ of the form   
\begin{equation}\label{linear_recur}
c_{0}b(t+n)+c_{1}b(t+n-1)+\cdots+c_{n}b(t)=0    
\end{equation}
for all $t\geq 0$. 

Suppose $N=n_{1}\cdots n_{k}$, then $\xi_{j}^{N}=1$ for all $j=1,\cdots,k$. Therefore, we have the generalized Fourier-Dedekind sum (and all other terms dependent on $\xi_{j}$) remaining invariant on replacing $t$ by $(t-qN)$ for any integer $q$, that is  
\begin{equation}\label{invariance_fds}
S_{t-qN}^{(m,r,p(x))}(\textbf{A}_{j};n_{j})=S_{t}^{(m,r,p(x))}(\textbf{A}_{j};n_{j}) ,  
\end{equation}
where $\textbf{A}_{j}=(n_{1},\cdots,\widehat{n_{j}},\cdots, n_{k})$, for all $j=1,\cdots,k$. With this observation at hand we can get rid of $O(t^{r-2})$ terms in (\ref{denum_recip}) employing the binomial identity 
\begin{equation}\label{bionomial_id}
\sum^{r-1}_{q=0}\begin{pmatrix}r-1 \\ q\end{pmatrix}(-1)^{q}(t-qN)^{s}=\left\{\begin{array}{@{}l@{\thinspace}l}
       0 & \textnormal{   if   } s<r-1 \\
       (r-1)!N^{r-1}  & \textnormal{   if   } s=r-1.
     \end{array}\right.
\end{equation}
   
\begin{figure}[t!]
\centering 
\includegraphics[scale=0.5]{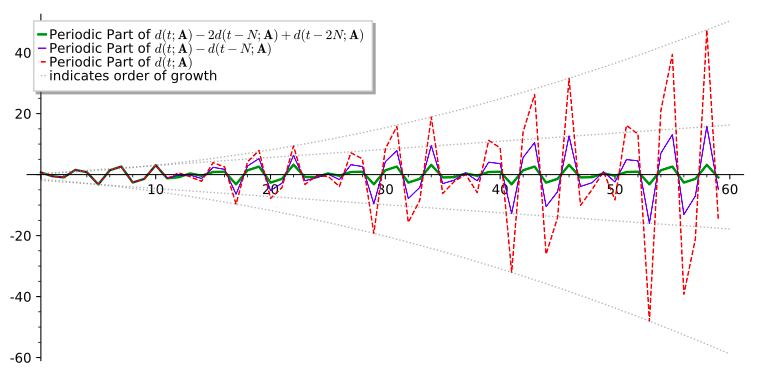}
\\ Here $\mathbf{A}=(3,3,3,4,4,4)$ and $N=12$. The amplitudes of the periodic parts of $d(t,;\mathbf{A})$, $d(t;\mathbf{A})-d(t-N;\mathbf{A})$ and $d(t;\mathbf{A})-2d(t-N;\mathbf{A})+d(t-2N;\mathbf{A})$ grow at quadratic, linear and constant rate respectively. 
\end{figure}  

As $W_{1}(t; \textbf{A})$ is a polynomial in $t$ of degree $(m+rk-1)$, and by means of (\ref{invariance_fds}) and (\ref{bionomial_id}), we can easily derive the equation  
\begin{equation}\label{higher_order_reciprocity}
b(t)=\sum^{r-1}_{q=0}\begin{pmatrix}r-1 \\ q\end{pmatrix}(-1)^{q}d(t-qN; \textbf{A})=\textnormal{poly}'(t)+N^{r-1}\sum^{k}_{j=1}S_{-t}^{(m,r,p(x))}(\textbf{A}_{j};n_{j}),
\end{equation}
where 
\begin{equation}\label{poly_W1}
\textnormal{poly}'(t)=\sum^{r-1}_{q=0}\begin{pmatrix}r-1 \\ q\end{pmatrix}(-1)^{q}W_{1}(t-qN;\textbf{A}).
\end{equation}
The above figure illustrates how the identity (\ref{bionomial_id}) reduces the order of growth in the amplitude yet retains the top-order term in the periodic part.

Moreover, the generating function for the left hand side of (\ref{higher_order_reciprocity}), $b(t)=\sum^{r-1}_{q=0}\begin{pmatrix}r-1 \\ q\end{pmatrix}(-1)^{q}d(t-qN;\textbf{A})$, is given by  
$$
G(x)=\sum^{\infty}_{t=0}b(t)x^{t}=\frac{p(x)(1-x^{N})^{r-1}}{(1-x)^{m}(1-x^{n_{1}})^{r}\cdots(1-x^{n_{k}})^{r}}.
$$
By \cite[Proposition 4.1.5]{Beck3}, $G(x)$ satisfies the linear recurrence relation (\ref{linear_recur}) if and only if $G(x)$ is a proper rational generating function. Therefore, we can derive a reciprocity relation provided $m$ is sufficiently large satisfying $m+r(n_{1}+\cdots+n_{k})>(r-1)N+d$ for $d=\textnormal{deg }(p)$. 

Setting $G^{o}(x)=-G(1/x)$ and by \cite[Theorem 4.1.6]{Beck3} we have 
$$
G^{o}(x)=\frac{(-1)^{m+r(k+1)}x^{\lambda}\tilde{p}(x)(1-x^{N})^{r-1}}{(1-x)^{m}(1-x^{n_{1}})^{r}\cdots(1-x^{n_{k}})^{r}}=\sum^{\infty}_{t=1}b(-t)x^{t},
$$
where $\lambda=m+r(n_{1}+\cdots+n_{k})-(r-1)N-d$ and $\tilde{p}(x)=x^{d}p(1/x)$. Here $b(-t)$ can be obtained by performing a backward computation for negative indices in (\ref{linear_recur}). For sufficiently large $m$ we have $\lambda>0$ and because of the factor $x^{\lambda}$ in the numerator of $G^{o}(x)$ we notice 
\begin{equation}\label{b_minus_t}
b(-t)=0,\quad \textnormal{ for }1\leq t < \lambda.
\end{equation} 
It is also easy to see that $b(0)=p(0)$. Substituting (\ref{b_minus_t}) into (\ref{higher_order_reciprocity}) we get
$$
\textnormal{poly}'(-t)+N^{r-1}\sum^{k}_{j=1}S_{t}^{(m,r,p(x))}(\textbf{A}_{j};n_{j})=\left\{\begin{array}{@{}l@{\thinspace}l}
         p(0) & \textnormal{   if   } t=0 \\
       0  & \textnormal{   if   }1\leq t<\lambda
     \end{array}\right.,
$$
where $\textnormal{poly}'(-t)=\sum^{r-1}_{q=0}\begin{pmatrix}r-1 \\ q\end{pmatrix}(-1)^{q}W_{1}(-t-qN;\textbf{A})$ is a polynomial in $t$ by (\ref{sec_W1}) and (\ref{comb_n}). Hence, we established the following result. 

\begin{thm}[Reciprocity Theorem] For sufficiently large $m\geq 0$ so that $\lambda>0$, where $\lambda=m+r(n_{1}+\cdots+n_{k})-(r-1)N-d$, the following equation holds
$$
\sum^{k}_{j=1}S_{t}^{(m,r,p(x))}(n_{1},\cdots,\widehat{n_{j}},\cdots,n_{k};n_{j}) =
\left\{\begin{array}{@{}l@{\thinspace}l}
         \frac{1}{N^{r-1}}(p(0)-\textnormal{poly}'(0)) & \textnormal{   if   } t=0 \\
       -\frac{1}{N^{r-1}}\textnormal{poly}'(-t)  & \textnormal{   if   }1\leq t<\lambda
     \end{array}\right.,
$$
where $\textnormal{poly}'(t)$ is given in (\ref{poly_W1}). 
\end{thm}

\section*{Acknowledgements}
I would like to thank Sampath Lonka, Rahul Raju Pattar and Hirak Doshi for discussions and support.



\begin{thebibliography}{10}
\providecommand{\url}[1]{{#1}}
\providecommand{\urlprefix}{URL }
\addtocounter{\@listctr}{0}
\makeatother
\expandafter\ifx\csname urlstyle\endcsname\relax
\providecommand{\doi}[1]{DOI~\discretionary{}{}{}#1}\else
\providecommand{\doi}{DOI~\discretionary{}{}{}\begingroup
	\urlstyle{rm}\Url}\fi

\bibitem{Alfonsin}
Alfons\'{i}n, J. L., The Diophantine Frobenius Problem, Oxford Lecture Series in Mathematics and its Applications. 
\newblock The Diophantine Frobenius Problem, Oxford Lecture Series in Mathematics and its Applications 30, Oxford University Press Inc. (2005).

\bibitem{Agnarsson}
Agnarsson, G., On the Sylvester denumerants for general restricted partitions, Congressus numerantium.
\newblock Congressus numerantium, 49--60 (1998).

\bibitem{BB}
Baldoni, V., Berline, N., De Loera, J., Dutra, B., K\"{o}ppe, M., Vergne, M.: Top degree coefficients of the denumerant.
\newblock DMTCS proc. AS., 1149--1160 (2013).

\bibitem{Beck0}
Beck, M., Gessel, I., Komatsu, T., The polynomial part of a restricted partition function related to the Frobenius problem.
\newblock The Electronic Journal of Combinatorics \textbf{8}(1), N7 (2001).

\bibitem{Beck1}
Beck, M., Diaz, R., and Robins, S., The Frobenius problem, rational polytopes, and Fourier–Dedekind sums.
\newblock J. Number Theory  \textbf{96}(1), 1--21 (2002).


\bibitem{Beck}
Beck, M., Robins, S.: Computing the continuous discretely: Integer-point enumeration in polyhedra.
\newblock Undergraduate Texts in Mathematics. Springer, New York, (2007).

\bibitem{Beck3}
Beck, M.,  Sanyal, R., Combinatorial Reciprocity Theorems.
\newblock American Mathematical Society, (2018).


\bibitem{Carlitz1}
Carlitz, L., The reciprocity formula for Dedekind sums.
\newblock Pacific J. Math.  \textbf{3}, 523--527 (1953).

\bibitem{degenerate_bernoulli}
Carlitz, L., Degenerate Striling, Bernoulli and Eulerian numbers.
\newblock Utilitas Math.  \textbf{15}, 51--88 (1979).

\bibitem{Mircea}
Cimpoea\c{s}, M., Nicolae, F., On the restricted partition function.
\newblock The Ramanujan Journal  \textbf{47}, 565--588 (2018).

\bibitem{Dickson}
Dickson, L.E., History of the Theory of Numbers. Volume II: Diophantine Analysis.
\newblock The Carnegie Institute, Washington, D.C. (1919).

\bibitem{Dilcher}
Dilcher, K., Vignat, C., An explicit form of the polynomial part of a restricted partition function.
\newblock Research in Number Theory  \textbf{3}(1), 12 (2017).

\bibitem{Duran}
Duran, A.: A sequence of polynomials related to roots of unity, Problem E 3339.
\newblock Amer. Math. Monthly  
\textbf{98}, 269--271 (1991).

\bibitem{Glaiser} Glaiser, J., On the number of partitions of a number into a given number of parts,.
\newblock Quart. J. Pure Appl. Math. \textbf{40}, 57–143 (1909).

\bibitem{Gessel}
Gessel, I., Generating functions and generalized Dedekind sums.
\newblock Electron. J. Combin.  \textbf{4}(2), 11--17 (1997).

\bibitem{gessel_umbral} 
Gessel, I., Application of the classical umbral calculus.
\newblock Algebra Universalls, \textbf{49}, 397-434 (2003). 

\bibitem{Ramanujan} Hardy, G., Ramanujan, S., Asymptotic formulae in combinatory analysis.
\newblock Proc. London. Math. Soc., \textbf{17}(2), 75-115 (1918).

\bibitem{laudano} Laudano, F., A generalization of the remainder theorem and factor theorem,
\newblock International Journal of Mathematical Education in Science and Technology, \textbf{50}, 960-967 (2019).

\bibitem{man} Man, Y., A cover-up approach to partial fractions with linear or irreducible quadratic factors in the denominators, 
\newblock Applied Mathematics and Computation \textbf{219}, 3855-3862 (2012).

\bibitem{munagi2}
Munagi, A., The Rademacher conjecture and $q$-partial fractions.
\newblock The Ramanujan Journal  \textbf{15}(2), 339--347 (2008).

\bibitem{Sullivan}
O'Sullivan, C., Partition and Sylvester Waves.
\newblock The Ramanujan Journal  \textbf{47}, 339--381 (2018).

\bibitem{Rubin}
Rubinstein, B., Fel, L., Restricted partition functions as Bernoulli and Eulerian polynomials of higher order.
\newblock The Ramanujan Journal  \textbf{11}, 331--347 (2006).

\bibitem{Rubinstein1}
Rubinstein, B., Expression for restricted partition function through Bernoulli polynomials.
\newblock The Ramanujan Journal  \textbf{15}, 117--185 (2008).

\bibitem{Sills}
Sills, A., Zeilberger D., Formulae for the number of partitions of $n$ into at most $m$ parts.
\newblock Adv. in Appl. Math. \textbf{48}, 640--645 (2012).

\bibitem{Sylvester}
Sylvester, J.J., On the partition of numbers.
\newblock Q. J. Math. \textbf{1}, 141--152 (1857).


\bibitem{Tuskerman}
Tsukerman E., Fourier-Dedekind sums and an extension of Rademacher reciprocity.
\newblock The Ramanujan Journal \textbf{37}, 421--460 (2015).

\bibitem{uk_sagemath}
Uday Kiran, N., Computing denumerants through $q$-partial fractions: Implementation in SageMath computer algebra system, GitHub repository, (2021). \hspace{0.3cm} https://github.com/nudaykiran/Denumerant/blob/main/denumerant.sagews  

\bibitem{uksl} Uday Kiran, N., Sampath Lonka, An algebraic approach to degenerate Bernoulli numbers, arXiv:2110.14100[math.NT] (2021).

\bibitem{Xin}
Xin, G., A fast algorithm for MacMahon's Partition Analysis.
\newblock The Electronic Journal of Combinatorics \textbf{11}(1), R58 (2004).


\bibitem{young}
Young, P., Degenerate Bernoulli polynomials, generalized factorial sums, and their applications.
\newblock Journal of Number Theory \textbf{128}(4), 738--758 (2008).

\bibitem{Zagier}
Zagier, D., Higher dimensional Dedekind sums.
\newblock Math. Ann. \textbf{202}, 149--172 (1973).

\bibitem{Zhang}
Zhang, Z., Yang, J., On sums of products of the degenerate Bernoulli numbers.
\newblock Integral Transforms and Special Functions \textbf{20}(10), 751--755 (2009).

\end{thebibliography}
\end{document}